\documentclass[a4paper,10pt,fleqn]{article}
\usepackage{amsmath,amssymb,amsthm,graphicx,subfigure,float,caption,epstopdf,tabularx,color, bm,amsfonts,epic}
\usepackage[top=1.25in, bottom=1.0in, left=1.0in, right=1.0in]{geometry}
\usepackage{appendix}
\usepackage{amssymb}
\usepackage{multirow}
\usepackage{mathrsfs}
\usepackage[table]{xcolor}
\usepackage[numbers,sort&compress]{natbib}
\usepackage{comment,enumerate,multicol,xspace}
\usepackage{fancyhdr}
\usepackage{soul}
\usepackage{enumerate}
\usepackage{makecell}
\definecolor{Cobalt}{rgb}{0.25,0.41,0.88}
\usepackage[linkcolor=Cobalt,anchorcolor=Cobalt,citecolor=Cobalt,colorlinks]{hyperref}
\usepackage{bbm}
\usepackage{algorithm}
\usepackage{algorithmic}
\usepackage{booktabs}
\usepackage{lineno}
\usepackage{dsfont}
\usepackage{cases}
\usepackage{pifont}

\allowdisplaybreaks
\newtheorem{thm}{Theorem}[section]
\newtheorem{lem}{Lemma}[section]

\newtheorem{den}{Definition}[section]

\newtheorem*{prf}{Proof}
\numberwithin{equation}{section}

\usepackage{indentfirst}
\graphicspath{{Fig}}

\begin{document}
\title{Finite difference methods for three kinds of reaction-diffusion equations with free boundaries}

\author{Caiyun Huang$^1$,\quad Weizhi Liao$^1$,\quad Weiping Bu$^{1,2}$\footnote{Correspondence author. Email: weipingbu@xtu.edu.cn}
\\{\small $^1$ School of Mathematics and Computational Science, Xiangtan University, }
\\{\small Hunan 411105, China}
\\{\small $^2$ Hunan Key Laboratory for Computation and Simulation in Science and}
\\{\small Engineering, Hunan 411105, China}
}

\date{}
\maketitle
\begin{abstract}
This work considers to numerically solve three kinds of reaction-diffusion equations with free boundaries.
First, the popular front-fixing method is used to transform the considered free boundary problems into fixed boundary problems.
Then, by employing the finite difference method, numerical schemes with $\mathrm{M}$-matrices as their coefficient matrices are developed for these transformed fixed-boundary problems.
Next, for the developed numerical schemes, we establish numerical theory involving positivity preservation, monotonicity preservation, and stability.
It is noteworthy that, unlike some existing works that impose tight restrictions on the time step-size to discuss the numerical theory, the proposed numerical schemes require only a mild restriction.
Finally, numerical examples are provided to test the
developed numerical schemes and to confirm the theoretical results.
\\[2ex]
\textbf{AMS subject classification:} 65M06, 65M12, 65M60.\\[2ex]
\textbf{Keywords:} Reaction-diffusion equation with free boundary; finite difference method; positivity preservation; monotonicity preservation; stability.
\end{abstract}


 \vskip 5mm
\section{Introduction}
\label{Sec1}
It is well known that differential equations are widely used in practical problems. Among them, differential equations with boundary value constraints are an important class, and such equations are typically defined on a preassigned known domain.
In 1891, however, Stefan\cite{stefan1891theorie} provided a classical example in his study of ice formation in polar seas, where the spatial domain is not known in advance but evolves with time, and thus must be determined together with the solution.
Generally speaking, problems in which the domain is not predetermined are  called free boundary problems.
At present, free boundary problems have already been involved in a wide range of fields in science and engineering\cite{crank1984free,chen2015free,carinci2016free,bansch2023interfaces}.

Recently, mathematical models with free boundaries have become a popular tool for simulating population spread.
To better understand the spreading mechanisms of new species and invasive species, some researchers have introduced Stefan conditions into population spread models.
Du and Lin\cite{du2010spreading} are pioneers among these researchers. They proposed a diffusive logistic model with free boundary, and established the spreading-vanishing dichotomy for this model. As is well-known, population models often need to consider general monostable nonlinear terms, while many models also adopt bistable nonlinearities to capture the Allee effect.
In \cite{kaneko2011free}, Kaneko and Yamada studied a class of reaction-diffusion equations with free boundary, and established the spreading-vanishing dichotomy for the solutions under both monostable and bistable nonlinearities.
Du and Lou\cite{du2015spreading} investigated a class of nonlinear diffusion equations with free boundary, and obtained the long-time behavior of the solutions based on monostable, bistable, and combustion types of nonlinearities, respectively. Under the radial symmetry assumption, Du and Guo\cite{du2011spreading} discussed the diffusive logistic equation with free boundary in higher space dimensions, and obtained the corresponding spreading-vanishing dichotomy and spreading speed. Subsequently, Du \emph{et al.}\cite{Du&Matsuzawa2015Spreading} conducted a detailed study of nonlinear diffusion problems with free boundaries and Stefan conditions was carried out for the cases of monostable, bistable, and combustion-type nonlinearities.

In this paper, we consider to numerical solve the following three kinds of reaction-diffusion equations with free boundaries.
The first kind has a right free boundary, which is described as
\begin{align}\label{eqs1_1}
\begin{cases}
u_t = d u_{xx} + f(u), \ 0 < x < h(t), \ t > 0, \\
u_x(t,0) = 0,\ u(t,h(t)) = 0, \ t > 0, \\
h'(t) = -\mu u_x(t,h(t)), \ t > 0, \\
h(0) = h_0, \ u(0,x) = u_0(x), \ 0 \leq x \leq h_0.
\end{cases}
\end{align}
where the initial function $u_0(x) \in C^2([0,h_0])$, $u_0'(0) = 0$, $u_0(h_0) = 0$, and $u_0(x) > 0$ for $0 \le x < h_0$, the free boundary $x = h(t)$ denotes the expanding front, and the homogeneous Neumann condition at $x = 0$ indicates that the left boundary is fixed and the population is confined to the region on its right.
The second kind involves the following double free boundary problem
\begin{align}\label{eqs1_2}
\begin{cases}
{u}_t = d {u}_{xx} + f({u}), \ {g}(t) < x < {h}(t), \ t > 0, \\	
{u}(t, {g}(t)) = 0,\ u(t, {h}(t)) = 0, \  t > 0, \\	
{g}'(t) = -\mu_1 {u}_x(t, {g}(t)), \  t > 0, \\			
{h}'(t) = -\mu_2 {u}_x(t, {h}(t)), \ t > 0, \\			
{g}(0) = -{h}_0,\ {h}(0) = {h}_0, \ {u}(0, x) = {u}_0(x), \  -{h}_0 \leq x \leq {h}_0,		
\end{cases}
\end{align}
where the initial function
${u}_0(x) \in C^2([-{h}_0,{h}_0]),  {u}_0(-{h}_0) = {u}_0({h}_0) = 0,$ and ${u}_0(x) > 0$ for $-{h}_0 \le x < {h}_0$. As the last kind,
we consider the following radially symmetric free boundary problem in two-dimensional polar coordinates
\begin{align}\label{eqs1_3}
\begin{cases}
u_t - d\left( u_{rr} + \frac{1}{r} u_r\right) = f( u), \ 0 < r <  h(t) ,\ t > 0, \\
u_r(t,0) = 0,\ u(t,h(t)) = 0, \ t > 0, \\
h'(t) = -\mu u_r(t,h(t)), \ t > 0, \\
h(0) = h_0,\ u(0,r) = u_0(r), \ 0 \le r \le h_0,
\end{cases}
\end{align}
where $r$ denotes the radial coordinate, the initial function $u_0(r) \in C^2([0,h_0])$, $u_0'(0) = 0$, $u_0(h_0) = 0$, and $u_0(r) > 0$ for $0 \le r < h_0$. Furthermore, for the above three kinds of free boundary problems, we assume that the parameters $d$, $\mu, \mu_1, \mu_2$, $h_0$ are positive constants, and there exists positive constant $L$ such that
\begin{equation}\label{eqs1_4}
-Lu\leq f(u)\leq Lu,\ \left| f'(u)\right|\leq L.
\end{equation}

For mathematical models involving free boundary conditions, their analytical solutions are often difficult to obtain in some complex situations.
Therefore, a large number of researchers have focused on the development of numerical methods to solve such problems\cite{crank1984free,caginalp1986analysis,marshall1986front,womble1989front,hou1995numerical,duffy2013finite}.
Recently, several numerical approaches have been proposed for reaction-diffusion population models with free boundaries arising in ecology.
In Ref. \cite{piqueras2017front}, Piqueras et al. applied the front-fixing method to transform the diffusive logistic population model with free boundary   into a fixed boundary problem, and proposed an explicit finite difference scheme to solve this problem.
Kumar and Rajeev\cite{kumar2020moving} later extended the front-fixing method to a space-fractional diffusive logistic model with free boundary.
Nevertheless, it is worth pointing out that efficient numerical methods for reaction-diffusion equations with free boundaries remain a challenging task.
In recent, Liu et al.\cite{liu2020krylov} employed the front-fixing method to transform a free boundary problem into a fixed boundary problem. And then, they  used four different time discretization methods, namely, the Runge-Kutta method, the Crank-Nicolson method, the IIF method, and the Krylov IIF method, to numerically solve the transformed problem, and compared the stability and computational efficiency of these four methods.
In Ref. \cite{liu2020numerical}, Liu et al. considered numerical methods for two-dimensional reaction-diffusion equations with Stefan-type free boundary conditions. In the radially symmetric case, they employed both the front-tracking method and the front-fixing method, whereas in the general case, they used the level set method.
In Ref. \cite{casaban2023qualitative}, Casab\'an et al. combined the front-fixing method with an explicit finite difference scheme to numerically solve a two-dimensional free-boundary diffusive logistic equation that is radially symmetric but heterogeneous.
Based on the front-fixing method, Yang and Bao \cite{yang2019numerical} constructed an efficient explicit finite difference scheme for a one-dimensional logistic chemotaxis model.
For a random free-boundary diffusive logistic model, Casab\'an et al. \cite{casaban2024random} developed finite difference schemes based on both the front-fixing and front-tracking methods.

It can be seen from the above discussion that numerical methods for reaction-diffusion models with free boundaries are still quite limited, especially the establishment of their numerical theory.
Thus, this fact motivates us to consider the problems \eqref{eqs1_1}--\eqref{eqs1_3}. In this paper, we aim to develop efficient finite difference methods to solve \eqref{eqs1_1}--\eqref{eqs1_3} based on the front-fixing method.
It is worth noting that, in existing related works such as Refs. \cite{kumar2020moving,piqueras2017front,casaban2023qualitative,yang2019numerical}, the positivity of the numerical solution, the monotonicity of the numerical solution for the spreading front, and the stability of the numerical scheme are guaranteed only when the spatial step-size $\Delta z$ is sufficiently small and the temporal step-size satisfies $\Delta t \le C(\Delta z)^2$ for some positive constant $C$. To overcome this extremely stringent
limitations, we propose implicit finite difference methods to solve problems \eqref{eqs1_1}--\eqref{eqs1_3}. The main contributions of this work are as follows:
1. For the three fixed boundary problems obtained by transforming \eqref{eqs1_1}--\eqref{eqs1_3} via the front-fixing method, three implicit finite difference schemes are constructed whose coefficient matrices are $\mathrm{M}$-matrices;
2. In contrast to the existing works \cite{kumar2020moving,piqueras2017front,casaban2023qualitative,yang2019numerical}, by using the properties of $\mathrm{M}$-matrix, we establish the positivity of the numerical solution and the monotonicity of the numerical solution for the moving fronts without the limitation that the spatial step-size $\Delta z$ is sufficiently small and  the time step-size $\Delta t$ depends on the spatial step-size $\Delta z$.
3. By devising new discrete weighted norms, the stability of the developed numerical schemes is established.
4. Some numerical tests are provided to verify the correctness of the obtained numerical theory.

The structure of this paper is as follows. In Section 2, we employ the front-fixing method to convert the free boundary problems \eqref{eqs1_1}--\eqref{eqs1_3} into the fixed boundary problems, and then construct the corresponding implicit finite difference schemes. In Section 3, the positivity of the numerical solution, the monotonicity of the numerical solution for the moving fronts, and the stability of the numerical schemes are derived in detail. In Section 4, numerical tests are carried out to verify the effectiveness of the proposed numerical methods and the correctness of the obtained numerical theory. Finally, a summary is provided in Section 5.

\textbf{Notation}. In this paper, we suppose that $C$ is a  constant and can be different in different situations.

\section{Construction of finite difference schemes}\label{Sec2}
In this section, the Landau transformation is used to convert the problems \eqref{eqs1_1}--\eqref{eqs1_3} into the equivalent problems on fixed domains firstly. Subsequently, implicit finite difference schemes are constructed for these equivalent problems.

To transform the considered free boundary problems \eqref{eqs1_1}--\eqref{eqs1_3} into problems on a fixed domain, we introduce different spatial variables $z,\tilde{z},\hat{z}$ and the corresponding functions $w,\tilde{w},\hat{w}$, respectively, so that the three transformed problems can be more clearly distinguished in the subsequent discussion.
Based on the Landau transformation~\cite{landau1950heat}, we introduce
$
z(t,x) = \frac{x}{h(t)}, \ w(t,z) = u(t,x).
$
Then the problem \eqref{eqs1_1} can be transformed into the following problem on the fixed interval $[0,1]$
\begin{equation}\label{eqs2_4}
\begin{cases}
r(t) w_t - r'(t) \frac{z}{2} w_z - d w_{zz} = r(t) f(w), \ 0<z<1,\ t>0,\\
w_z(t,0)=0,\ w(t,1)=0, \ t>0,\\
r'(t) = -2\mu\, w_z(t,1), \ t>0,\\
r(0)=h_0^2,\ w(0,z)=w_0(z), \ 0\le z\le 1,
\end{cases}
\end{equation}
where $r(t) = h^2(t)$, the initial function $w_0(z) = u_0(zh_0)$ satisfies $w_0(z) \in C^2([0, 1])$, $w_0'(0) = w_0(1) = 0$, and $w_0(z) > 0$ for $0 \leq z < 1$.
Similarly, by introducing
$
\tilde{z}(t,x)=\frac{x - {g}(t)} {{h}(t) - {g}(t)},\ \tilde{w}(t,z) = {u}(t, x),
$
the problem \eqref{eqs1_2} can be transformed into
\begin{equation}\label{eqs2_5}
\begin{cases}
\tilde w_t - \frac{(1-\tilde z) g'(t) + \tilde z  h'(t)}{\tilde h(t) - g(t)} \tilde w_{\tilde z} - \frac{d}{( h(t) -  g(t))^2} \tilde w_{\tilde z\tilde z} = f(\tilde w), \ 0 < \tilde z < 1, \ t > 0, \\
\tilde w(t,0) = 0, \ \tilde w(t,1) = 0, \ t > 0, \\
 g'(t) = -\frac{\mu_1}{ h(t) - g(t)} \tilde w_{\tilde z}(t,0), \ t > 0, \\
h'(t) = -\frac{\mu_2}{ h(t) - g(t)} \tilde w_{\tilde z}(t,1), \ t > 0, \\
g(0)=- h_0, \  h(0) =  h_0, \ \tilde w(0, \tilde z) = \tilde w_0(\tilde z), \ 0 \leq \tilde z \leq 1,
\end{cases}
\end{equation}
where the initial function $\tilde w_0(\tilde z) =  u_0(\tilde z h_0)$ satisfies $\tilde w_0(\tilde z) \in C^2([0, 1])$, $\tilde w_0(0) = \tilde w_0(1) = 0$, and $\tilde w_0(\tilde z) > 0$ for $0 < \tilde z < 1$.
Furthermore, we replace $r,u$ in \eqref{eqs1_3} by
$
\hat z(t,r) = \frac{r}{ h(t)}, \hat w(t,z) = u(t,r),
$
yielding
\begin{equation}\label{eqs2_6}
\begin{cases}
\hat r(t)\hat z \hat w_t - d\hat z \hat w_{\hat z \hat z}
- \left( d + \frac{\hat z^2 \hat r'(t)}{2} \right) \hat w_{\hat z}
= \hat z\hat r(t)f(\hat w), \ 0<\hat z<1,\ t>0, \\
\hat w_{\hat z}(t,0) = 0,\ \hat w(t,1) = 0, \ t>0,\\
\hat r'(t) = -2\mu \hat w_{\hat z}(t,1), \ t>0, \\
\hat r(0) =  h_0^2,\ \hat w(0,\hat z) = \hat w_0(\hat z), \ 0 \le \hat z \le 1,
\end{cases}
\end{equation}
where $\hat r(t) =  h^2(t)$, the initial function $\hat w_0(\hat z) =  u_0(\hat z h_0)$ satisfies $\hat w_0(\hat z) \in C^2([0, 1])$, $\hat w_0'(0) = \hat w_0(1) = 0$, and $\hat w_0(\hat z) > 0$ for $0 \leq \hat z < 1$.

For any fixed $T>0$, we now consider to construct finite difference schemes for problems \eqref{eqs2_4}--\eqref{eqs2_6} on the time interval $[0,T]$.
Let $\Delta \xi$ and $\Delta t$ be the spatial and temporal step-sizes, respectively, such that $M\Delta \xi=1$ and $N\Delta t=T$ for positive integers $M$ and $N$. Denote
$\xi_j=j\Delta \xi,\ j=0,1,\cdots,M,\
t_n=n\Delta t,\ n=0,1,\cdots,N,
\Omega_{s}=\{\xi_j\mid \xi_j=j\Delta \xi,\ 0\le j\le M\},\
\Omega_t=\{t_n\mid t_n=n\Delta t,\ 0\le n\le N\}.
$
Then, for any mesh function $U=\{U_j^n\mid 0\le j\le M,\ 0\le n\le N\}$ on $\Omega_{s}\times\Omega_t$, and any mesh function $Q=\{Q^n\mid 0\le n\le N\}$ on $\Omega_t$, we define the following temporal and spatial difference operators
\begin{equation}\label{eqs2_7}
\delta_t U_j^n=\frac{U_j^n-U_j^{n-1}}{\Delta t},\quad
\delta_t Q^n=\frac{Q^n-Q^{n-1}}{\Delta t},
\end{equation}
\begin{equation}\label{eqs2_7a}
\delta_{\xi} U_{j+\frac12}^n = \frac{U_{j+1}^n - U_j^n}{\Delta {\xi}},\quad
\delta_{\xi}^2 U_j^n =
\begin{cases}
\frac{2}{\Delta {\xi}}\delta_{\xi} U_{\frac12}^n, & j=0,\\
\frac{1}{\Delta {\xi}}\left(\delta_{\xi} U_{j+\frac12}^n - \delta_{\xi} U_{j-\frac12}^n\right), &j\neq 0,
\end{cases}
\end{equation}
where $j=0,1,\cdots,M-1$, $n=1,2,\cdots,N$. It should be pointed out that, when the problems \eqref{eqs2_4}--\eqref{eqs2_6} are considered, we can obviously replace the above variable $\xi$ by $z,\tilde{z}$ and $\hat{z}$, respectively.
Let $w_j^n=w(t_n,z_j)$ and $r^n=r(t_n).$  By employing \eqref{eqs2_7} and \eqref{eqs2_7a}, then the problem \eqref{eqs2_4} gives
\begin{equation}\label{eqs2_9}
\begin{cases}
r^n\delta_t w_0^n - {d}\delta_z^2 w_0^n = r^n f(w_0^{n-1}) + T_0^n, \ 1\leq n\leq N ,\\
r^n\delta_t w_j^n - \frac{z_j\delta_t r^n}{2}\delta_z w_{j+\frac12}^n  - d\,\delta_z^2 w_j^n = r^n f(w_j^{n-1}) + T_j^n,\ 1\le j\le M-1,\ 1\leq n\leq N,\\
\delta_t r^n = - {2\mu}\delta_z w_{M-\frac{1}{2}}^{n-1} + T_M^n,\ 1\leq n\leq N,\\
w_M^n=0,\ 1\leq n\leq N,
\end{cases}
\end{equation}
where $T_j^n$ is the corresponding local truncation error.
Analogously, let $\tilde{w}_j^n=\tilde{w}(t_n,\tilde{z}_j), g^n=g(t_n), h^n=h(t_n)$ and $\hat{w}_j^n=\hat{w}(t_n,\hat{z}_j), \hat{r}^n=\hat{r}(t_n)$, applying \eqref{eqs2_7} and \eqref{eqs2_7a} to the problems \eqref{eqs2_5} and \eqref{eqs2_6} leads respectively to
\begin{equation}\label{eqs2_10}
\begin{cases}
\delta_t \tilde w_j^n
- \frac{(1-{\tilde z}_j)\,\delta_t g^n}{h^n - g^n}\,\delta_{\tilde z} \tilde w_{j-\frac12}^n
- \frac{\tilde z_j\,\delta_t h^n}{ h^n - g^n}\,\delta_{\tilde z} \tilde w_{j+\frac12}^n
- \frac{d}{( h^n - g^n)^2}\,\delta_{\tilde z}^2 \tilde w_j^n\\
= f(\tilde w_j^{n-1})+\tilde T_j^n,\ 1\le j \le M-1,\ 1\le n \le N,\\
\tilde w_0^n = 0,\ \tilde w_M^n = 0, \ 1\le n \le N,\\
\delta_t g^n = -\frac{\mu_1}{ h^{n-1} - g^{n-1}}\,\delta_{\tilde z} \tilde w_{\frac12}^{n-1}+\tilde T_0^n, \ 1\le n \le N,\\
\delta_t h^n = -\frac{\mu_2}{ h^{n-1} - g^{n-1}}\,\delta_{\tilde z} \tilde w_{M-\frac12}^{n-1}+\tilde T_M^n,\ 1\le n \le N,
\end{cases}
\end{equation}
and
\begin{equation}\label{eqs2_11}
\begin{cases}
\hat r^n\hat z_j\delta_t \hat w_j^n - d\hat z_j\delta_{\hat z}^2 \hat w_j^n - \left( d + \frac{\hat z_j^2 \delta_t \hat r}{2} \right)\delta_{\hat z} \hat w_{j+\frac12}^n\\
= \hat r^n\hat z_j f(\hat w_j^{n-1}) + \hat T_j^n,\ 1\le j\le M-1,\ 1\leq n\leq N ,\\
\delta_{\hat z} \hat w_{\frac12}^n=\hat T_0^n,\ \hat w_M^n=0,\ 1\leq n\leq N, \\
\delta_t \hat r^n = - {2\mu}\delta_{\hat z}\hat w_{M-\frac{1}{2}}^{n-1} + \hat T_M^n,\ 1\leq n\leq N,
\end{cases}
\end{equation}
where
$\tilde T_j^n$ and $\hat T_j^n$ denote the corresponding local truncation errors.
What is more, if
$w,\tilde w,\hat w \in C^{4,2}([0,1]\times[0,T])$
and
$r, g, h, \hat r\in C^2[0,T]$,
then the classical local truncation error theory implies that there exists a positive constant $C$ such that
$
|T_j^n|\le C(\Delta t+\Delta z),
|\tilde T_j^n|\le C(\Delta t+\Delta \tilde z),
|\hat T_j^n|\le C(\Delta t+\Delta \hat z).
$
Now, omitting the truncation errors in \eqref{eqs2_9}--\eqref{eqs2_11}, we obtain the implicit finite difference schemes for problems \eqref{eqs2_4}--\eqref{eqs2_6} as follows
\begin{subequations}\label{eqs2_12}
\begin{numcases}{}
R^n \delta_t W_j^n - d\delta_z^2 W_j^n = R^n f(W_j^{n-1}), \ j=0 , \ 1\leq n\leq N,\label{eqs2_12.a}\\
R^n \delta_t W_j^n - \frac{z_j\delta_t R^n}{2}\delta_z W_{j+\frac12}^n  - d\,\delta_z^2 W_j^n = R^n f(W_j^{n-1}), \ 1\le j\le M-1, \ 1\leq n\leq N,\label{eqs2_12.b}\\
\delta_t R^n = -{2\mu}\delta_z W_{M-\frac{1}{2}}^{n-1},  \ 1\leq n\leq N,\label{eqs2_12.c}\\
W_M^n = 0,  \ 1\leq n\leq N,\label{eqs2_12.d}
\end{numcases}
\end{subequations}
\begin{subequations}\label{eqs2_13}
\begin{numcases}{}
\delta_t \tilde W_j^n
- \frac{(1-\tilde{z}_j)\,\delta_t G^n}{ H^n - G^n}\,\delta_{\tilde{z}} \tilde W_{j-\frac12}^n
- \frac{\tilde{z}_j\,\delta_t H^n}{H^n - G^n}\,\delta_{\tilde{z}} \tilde W_{j+\frac12}^n
- \frac{d}{( H^n - G^n)^2}\,\delta_{\tilde{z}}^2 \tilde W_j^n\notag \\
= f(\tilde W_j^{n-1}),\ 1\le j \le M-1,\ 1\le n \le N,\label{eqs2_13.a}\\
\tilde W_0^n = 0,\ \tilde W_M^n = 0, \ 1\le n \le N,\label{eqs2_13.b}\\
\delta_t G^n = -\frac{\mu_1}{H^{n-1} - G^{n-1}}\,\delta_{\tilde{z}} \tilde W_{\frac12}^{n-1}, \ 1\le n \le N,\label{eqs2_13.c}\\
\delta_t H^n = -\frac{\mu_2}{H^{n-1} - G^{n-1}}\,\delta_{\tilde{z}} \tilde W_{M-\frac12}^{n-1},\ 1\le n \le N,\label{eqs2_13.d}
\end{numcases}
\end{subequations}\
and
\begin{subequations}\label{eqs2_14}
\begin{numcases}{}
\hat R^n\hat z_j\delta_t \hat W_j^n - d\hat z_j\delta_{\hat z}^2 \hat W_j^n - \left( d + \frac{\hat z_j^2 \delta_t\hat R^n}{2} \right)\delta_{\hat z} \hat W_{j+\frac12}^n
= \hat R^n\hat z_j f(\hat W_j^{n-1}),\ 1\le j\le M-1,\ 1\leq n\leq N ,\label{eqs2_14.a}\\
\delta_{\hat z} \hat W_{\frac12}^n=0,\ \hat W_{M}^n=0,\ 1\leq n\leq N, \label{eqs2_14.b}\\
\delta_t \hat R^n = - {2\mu}\delta_{\hat z}\hat W_{M-\frac{1}{2}}^{n-1}, \ 1\leq n\leq N.\label{eqs2_14.c}
\end{numcases}
\end{subequations}
\section{Positivity preservation, monotonicity preservation and stability analysis}
In order to establish the numerical theory of the proposed numerical schemes, we first introduce the $\mathrm{M}$-matrix and its some properties.
\begin{den}\label{def1}(\cite{berman1994nonnegative})
Let $A=(a_{ij})_{n\times n}$ satisfy $a_{ij}\le 0$ for all  $1\le i\ne j\le n$. If $A$ can be written in the form
$A=sI-B$, then $A$ is called a nonsingular $\mathrm{M}$-matrix and denoted by $A\in M_n$, where $B=(b_{ij})_{n\times n}$ with each element $b_{ij}\geq 0$ and $s>\rho(B)$ denoting the spectral radius of $B$ as $\rho(B)$.
\end{den}
\begin{den}\label{def2}(\cite{horn2012matrix})
Let $A=(a_{ij})_{n\times n}\in M_n$. The matrix
$A$ is said to have property SC if, for each pair of different integers $p,q\in\{1,2,\cdots,n\},$ there exists a sequence of different integers
$k_1=p,\ k_2,\ \cdots,\ k_m=q$
such that
$a_{k_1k_2},\ a_{k_2k_3},\ \cdots,\ a_{k_{m-1}k_m}$
are all nonzero.
\end{den}
\begin{lem}\label{lem1}(\cite{berman1994nonnegative,horn2012matrix})
Let $A\in M_n$. The matrix $A$ has property SC if and only if $A$ is irreducible.
\end{lem}
\begin{lem}\label{lem2}(\cite{berman1994nonnegative})
Let $A=(a_{ij})_{n\times n}$ with $a_{ij}\le 0$ for all $1\le i\ne j\le n$, and assume that $A$ is irreducible. Then $A^{-1}=(b_{ij})_{n\times n}$ with $b_{ij}>0$ for all $i,\ j=1,2,\cdots,n$ if and only if $A\in M_n$.
\end{lem}
\begin{lem}\label{lem3}
Let $A\mathbf{x}=\mathbf{b}$. If $A=(a_{ij})_{n\times n}$ is a tridiagonal matrix satisfying $a_{ij}<0$ for $|i-j|=1$ and $a_{ii}>\sum_{j\ne i}|a_{ij}|$ for $i=1,2,\dots,n$, and if $\mathbf{b}=(b_1,b_2,\ldots,b_n)^\top$ satisfies $b_i>0$ for $i=1,2,\dots,n$, then there exists a unique solution $\mathbf{x}=(x_1,x_2,\ldots,x_n)^\top$ satisfies $x_i>0$ for $i=1,2,\ldots,n$.
\end{lem}
\begin{prf}
According to the assumption, it is clear that $A$ can be written as
$A = sI - B,$
where
$s = \max_{1 \leq i \leq n} a_{ii}$,
and $B=(b_{ij})_{n\times1}$ satisfies $b_{ij}\ge 0$ for all $i,\ j=1,2,\cdots,n$.  We now show that \(s > \rho(B)\). Using the Gershgorin disk theorem, each eigenvalue of $B$ lies in the union of the Gershgorin disks
$
G(B) = \bigcup_{i=1}^n \{ z \in \mathbb{C}, |z - b_{ii}| \leq R_i(B) \},
$
where $R_i(B) = \sum_{j \neq i} |b_{ij}|, i = 1, 2,\cdots, n$.
Since $B=sI-A$, the centers of these disks are given by
$
b_{ii}=s-a_{ii}\ge 0,\ i=1,2,\dots,n,
$
and their radii are
\begin{equation*}
R_i(B)=\sum_{j\ne i}|b_{ij}|=
\begin{cases}
-a_{i,i+1}, \ i=1,\\
-a_{i,i-1}-a_{i,i+1}, \ 2\le i\le n-1,\\
-a_{i,i-1}, \ i=n.
\end{cases}
\end{equation*}
Therefore, the rightmost intersection point of the disk centered at $b_{ii}$ with radius $R_i(B)$ and the real axis is
\begin{equation*}
s-a_{ii}+R_i(B)=
\begin{cases}
s-(a_{ii}+a_{i,i+1}), \ i=1,\\
s-(a_{ii}+a_{i,i-1}+a_{i,i+1}), \ 2\le i\le n-1,\\
s-(a_{ii}+a_{i,i-1}), \ i=n.
\end{cases}
\end{equation*}
From $a_{ii}>\sum_{j\ne i}|a_{ij}|$, it follows that $s-a_{ii}+R_i(B)<s$,
it clearly implies $\rho(B)<s$. Therefore, from Definition \ref{def1}, $A$ is a nonsingular $\mathrm{M}$-matrix.
Moreover, since $A$ is tridiagonal and satisfies $a_{ij}<0$ for $|i-j|=1$, it follows from Definition \ref{def2} that $A$ has property SC. Thus, by Lemma \ref{lem1}, $A$ is irreducible. The above derivation shows that $A$ is an irreducible nonsingular $\mathrm{M}$-matrix. Now, by Lemma \ref{lem2}, we can conclude that there exists a unique solution $\mathbf{x}=A^{-1}\mathbf{b}$ of $A\mathbf{x}=\mathbf{b}$ satisfies $x_i>0$ for $i=1,2,\cdots,n$.
\end{prf}
\subsection{Positivity and monotonicity preservation}
Next, we discuss the positivity of the numerical solution $W_j^n$ and the monotonicity of $R^n$ for numerical scheme \eqref{eqs2_12}, the positivity of the numerical solution $\tilde W_j^n$ and the monotonicity of $G^n$ and $H^n$ for numerical scheme \eqref{eqs2_13}, and the positivity of the numerical solution $\hat W_j^n$ and the monotonicity of $\hat R^n$ for numerical scheme \eqref{eqs2_14}, respectively.

\begin{thm}
    For numerical scheme \eqref{eqs2_12}, if $\Delta t < 1/L$, then the numerical solution satisfies
\begin{equation}\label{eqs3_4}
    W_j^n > 0, \quad 1 \le n \le N, \quad 0 \le j \le M-1,
\end{equation}
and
\begin{equation}\label{eqs3_5}
    R^n > R^{n-1}, \quad 1 \le n \le N.
\end{equation}
\end{thm}

\begin{prf}
From \eqref{eqs2_12.a} and \eqref{eqs2_12.b}, we have
\begin{align}\label{eqs3_a}
\left( \dfrac{R^n}{\Delta t}+\dfrac{2d}{(\Delta z)^2}\right) W_0^n-\dfrac{2d}{(\Delta z)^2}W_1^n
= \dfrac{R^n}{\Delta t}W_0^{n-1}+R^n f(W_0^{n-1}), \  1 \le n \le N,
\end{align}
and
\begin{align}\label{eqs3_b}
&\left( \dfrac{R^n}{\Delta t}+\dfrac{z_j}{2}\dfrac{R^n-R^{n-1}}{\Delta t \Delta z}+\dfrac{2d}{(\Delta z)^2}\right) W_j^n
-\left( \dfrac{z_j}{2}\dfrac{R^n-R^{n-1}}{\Delta t\Delta z}+\frac{d}{(\Delta z)^2}\right) W_{j+1}^n
-\dfrac{d}{(\Delta z)^2}W_{j-1}^n \notag \\
= &\dfrac{R^n}{\Delta t}W_j^{n-1}+R^n f(W_j^{n-1}), \ 1\le j\le M-1,\ 1 \le n \le N.
\end{align}
Thus, combining \eqref{eqs3_a} with \eqref{eqs3_b} leads to
\begin{equation}\label{eqs3_8}
\bar{A}_n\mathbf{\bar{W}}_n = \mathbf{b}_n,\quad 1\le n\le N,
\end{equation}
where
$
\mathbf{\bar{W}}_n=\left(W_0^n,W_1^n,\ldots,W_{M-1}^n\right)^\top,
$
\[
\mathbf{b}_n=\left( \frac{R^n}{\Delta t}W_0^{n-1}+R^nf(W_0^{n-1}),\
\frac{R^n}{\Delta t}W_1^{n-1}+R^n f(W_1^{n-1}),\
\ldots,\
\frac{R^n}{\Delta t}W_{M-1}^{n-1}+R^n f(W_{M-1}^{n-1})\right)^\top,
\]
and the coefficient matrix $\bar{A}_n=(a_{ij}^n)_{M\times M}$ is tridiagonal with
\begin{equation}\label{eqs3_9}
\begin{cases}
a_{ij}^n=\frac{R^n}{\Delta t}+\frac{z_j}{2}\frac{R^n-R^{n-1}}{\Delta t\Delta z}+\frac{2d}{(\Delta z)^2}, \ i=j,\\
a_{ij}^n=-\frac{z_j}{2}\frac{R^n-R^{n-1}}{\Delta t\Delta z}-\frac{d}{(\Delta z)^2}, \ i=j-1,\\
a_{ij}^n=-\frac{d}{(\Delta z)^2}, \ i=j+1,\\
a_{ij}^n=0,\ |i-j|\ge 2.
\end{cases}
\end{equation}
For $n=1$, it follows from \eqref{eqs2_12.c} and \eqref{eqs2_12.d} that
$
R^1=R^0+\frac{2\mu \Delta t}{\Delta z}W_{M-1}^{0}.
$
Given that $W_j^0=u_0(z_jh_0)>0$, we get
\begin{equation}\label{eqs3_10}
R^1>R^0.
\end{equation}
By applying \eqref{eqs3_9} and  \eqref{eqs3_10}, we deduce that the matrix $\bar{A}_1$ in \eqref{eqs3_8} satisfies
\begin{equation}\label{eqs3_11}
a_{ii}^1>\sum_{j\neq i}|a_{ij}^1|,\ a_{ij}^1<0,\ |i-j|=1.
\end{equation}
Besides, the element of the vector $\mathbf{b}_1$ in \eqref{eqs3_8} is given by
$
b_j^1=\frac{R^1}{\Delta t}W_j^0+R^1 f(W_j^0).
$
Under the condition $\Delta t<1/L$, it follows from \eqref{eqs1_4} that
\begin{align}\label{eqs3_13}
b_j^1 &\ge \frac{R^1}{\Delta t}W_j^0-LR^1W_j^0 >0.
\end{align}
Thus, the combination of \eqref{eqs3_11} and \eqref{eqs3_13} and Lemma \ref{lem3} yields $W_j^1>0$ for $0\le j\le M-1$. It means that, for $n=1$, the inequalities  \eqref{eqs3_4} and \eqref{eqs3_5} hold.

Assume that, for some $l\ge 2$, the inequalities \eqref{eqs3_4} and \eqref{eqs3_5} hold for $1\le n\le l-1$ under the condition $\Delta t<1/L$. Now we consider the case $n=l$. According to \eqref{eqs2_12.c} and \eqref{eqs2_12.d}, one has
\begin{equation*}
R^l=R^{l-1}+\frac{2\mu \Delta t}{\Delta z}W_{M-1}^{l-1}.
\end{equation*}
Since \eqref{eqs3_5} holds for $n=l-1$, i.e., $W_j^{l-1}>0$ for $0\le j\le M-1$, we deduce that
\begin{equation}\label{eqs3_14}
R^l>R^{l-1}.
\end{equation}
From \eqref{eqs3_9} and \eqref{eqs3_14}, it is easy to verify that the matrix $\bar{A}_l$ in \eqref{eqs3_8} satisfies
\begin{equation}\label{eqs3_15}
a_{ii}^l > \sum_{j\neq i} |a_{ij}^l|,\ a_{ij}^l<0,\ |i-j|=1.
\end{equation}
For the vector $\mathbf{b}_l$ in \eqref{eqs3_8}, its element is given by
$
b_j^l=\frac{R^l}{\Delta t}W_j^{l-1}+R^l f(W_j^{l-1}).
$
If $\Delta t<1/L$, then \eqref{eqs1_4} implies
\begin{align}\label{eqs3_17}
b_j^l &\ge \frac{R^l}{\Delta t}W_j^{l-1}-LR^lW_j^{l-1}>0.
\end{align}
By combining \eqref{eqs3_15} and \eqref{eqs3_17} and Lemma \ref{lem3},
it is clear that $W_j^l>0, 0\le j\le M-1$. Therefore, via the mathematical induction, for any $1\le n\le N$, if $\Delta t<1/L$, then \eqref{eqs3_4} and \eqref{eqs3_5} hold.

\end{prf}
Next, we employ the analogous argument to analyze numerical schemes \eqref{eqs2_13} and \eqref{eqs2_14}.
\begin{thm} \label{thm3_1}
For the numerical scheme \eqref{eqs2_13}, if $\Delta t<1/L$,
then the numerical solution satisfies
\begin{equation}\label{eqs3_18}
    \tilde{W}_j^n>0,\ 1\le j\le M-1,\ 1\le n \le N,
\end{equation}
and
\begin{equation}\label{eqs3_19}
     G^n< G^{n-1},\  H^{n}> H^{n-1},\ 1\le n \le N.
\end{equation}
\end{thm}
\begin{prf}
From \eqref{eqs2_13.a}, it can be deduced that
    \begin{align*}
&\left[\frac{1}{\Delta t} - \frac{(1-\tilde{z}_j)( G^n -  G^{n-1})}{\Delta \tilde{z} \cdot \Delta t ( H^n -  G^n)} + \frac{\tilde{z}_j( H^n - H^{n-1})}{\Delta \tilde{z} \cdot \Delta t ( H^n - G^n)} + \frac{2d}{( H^n - G^n)^2}\right]\tilde{W}_j^n \notag\\
&+ \left[\frac{(1-\tilde{z}_j)( G^n -  G^{n-1})}{\Delta \tilde{z} \cdot \Delta t ( H^n - G^n)} - \frac{d}{( H^n - G^n)^2}\right] \tilde{W}_{j-1}^n
-\left[ \frac{\tilde{z}_j( H^n - H^{n-1})}{\Delta \tilde{z} \cdot \Delta t ( H^n - G^n)} + \frac{d}{( H^n - G^n)^2} \right] \tilde{W}_{j+1}^n \notag \\
= &\frac{1}{\Delta t} \tilde{W}_j^{n-1} + f(\tilde{W}_j^{n-1}), \quad 1 \le  j \le M-1, \ 1 \le n \le N.
\end{align*}
Rewriting the above equation in matrix form yields
\begin{equation}\label{eqs3_20}
\tilde A_n\tilde{\mathbf{W}}_n=\tilde{\mathbf{b}}_n,\ 1\le n \le N,
\end{equation}
where
$
\tilde{\mathbf{W}}_n=\left(\tilde{W}_1^n, \tilde{W}_2^n, \ldots, \tilde{W}_{M-1}^n \right)^\top,
$
$
\tilde{\mathbf{b}}_n=(\tilde{b}_j^n)_{(M-1)\times1}
$
with
$
\tilde{b}_j^n=\frac{1}{\Delta t} \tilde{W}_j^{n-1}+f(\tilde{W}_j^{n-1}),
$
and the coefficient matrix $\tilde A_n=(\tilde a_{ij}^n)_{(M-1)\times (M-1)}$ is tridiagonal with
\begin{equation}\label{eqs3_21}
    \begin{cases}
        \tilde a_{ij}^n=\frac{1}{\Delta t} - \frac{(1-\tilde{z}_j)( G^n - G^{n-1})}{\Delta \tilde{z} \cdot \Delta t (H^n - G^n)} + \frac{\tilde{z}_j(H^n - H^{n-1})}{\Delta \tilde{z} \cdot \Delta t (H^n - G^n)} + \frac{2d}{(H^n - G^n)^2}, \ i=j,\\
        \tilde a_{ij}^n=\frac{(1-\tilde{z}_j)(G^n - G^{n-1})}{\Delta \tilde{z} \cdot \Delta t (H^n - G^n)} - \frac{d}{(H^n - G^n)^2}, \ i=j+1,\\
        \tilde a_{ij}^n=-\frac{\tilde{z}_j(H^n - H^{n-1})}{\Delta \tilde{z} \cdot \Delta t (H^n - G^n)} - \frac{d}{( H^n - G^n)^2}, \ i=j-1,\\
        \tilde a_{ij}^n=0,\ |i-j|\ge 2.
    \end{cases}
\end{equation}
For $n=1$, it follows from \eqref{eqs2_13.b}, \eqref{eqs2_13.c} and \eqref{eqs2_13.d} that
\begin{equation*}
    G^1=G^0-\frac{\mu_1\Delta t}{\Delta \tilde{z}(H^0- G^0)}\tilde{W}_1^{0},\   H^1=H^0+\frac{\mu_2\Delta t}{\Delta \tilde{z}( H^0- G^0)}\tilde{W}_{M-1}^{0}.
\end{equation*}
Thus, $\tilde{W}_j^0= u_0(\tilde{z_j} h_0)>0$ implies that
\begin{equation}\label{eqs3_22}
     G^1< G^0,\  H^1> H^0.
\end{equation}
In view of \eqref{eqs3_21} and \eqref{eqs3_22}, the element of matrix $\tilde A_1$ in \eqref{eqs3_20} has the following property
\begin{equation}\label{eqs3_23}
 \tilde a_{ii}^1>\sum_{j\neq i}|\tilde a_{ij}^1|,\ \tilde a_{ij}^1<0,\ |i-j|=1.
\end{equation}
Besides, since $\tilde{\mathbf{b}}_1$ in \eqref{eqs3_20} is given by $\tilde b_j^1=\frac{1}{\Delta t} \tilde{W}_j^{0}+f(\tilde{W}_j^{0}),$
it follows from \eqref{eqs1_4} and $\Delta t<1/L$ that
\begin{align}\label{eqs3_25}
    \tilde b_j^1&\ge \frac{1}{\Delta t} \tilde{W}_j^{0}-L\tilde{W}_j^{0}>0.
\end{align}
By combining \eqref{eqs3_23} and \eqref{eqs3_25}, it yields $\tilde{W}_j^{1}>0$ for $1\le j\le M-1$. Hence, for $n=1$, if $\Delta t<1/L$, then \eqref{eqs3_18} and \eqref{eqs3_19} hold.

Suppose that, for some integer $l\ge 2$, \eqref{eqs3_18} and \eqref{eqs3_19} hold for all $1\le n\le l-1$ under the condition $\Delta t<1/L$. We now consider the case $n=l$. By \eqref{eqs2_13.b} and \eqref{eqs2_13.c}, one has
\begin{equation*}
 G^l= G^{l-1}-\frac{\mu_1\Delta t}{\Delta \tilde{z}( H^{l-1}- G^{l-1})}\tilde{W}_1^{l-1},\   H^l= H^{l-1}+\frac{\mu_2\Delta t}{\Delta \tilde{z}( H^{l-1}- G^{l-1})}\tilde{W}_{M-1}^{l-1}.
\end{equation*}
Noting that \eqref{eqs3_18} holds for $1\le n\le l-1$ under the condition $\Delta t<\frac{1}{L}$, we get
\begin{equation}\label{eqs3_26}
 G^l< G^{l-1},\  H^{l}> H^{l-1}.
\end{equation}
Via applying \eqref{eqs3_26} to \eqref{eqs3_21}, it leads to
\begin{equation}\label{eqs3_27}
\tilde a_{ii}^l>\sum_{j\neq i}|\tilde a_{ij}^l|,\     \tilde a_{ij}^l<0,\ |i-j|=1.
\end{equation}
For the vector $\tilde{\mathbf{b}}_l$ in \eqref{eqs3_20}, $\tilde b_j^l=\frac{1}{\Delta t}\tilde{W}_j^{\,l-1}+f\bigl(\tilde{W}_j^{\,l-1}\bigr)$
and $\Delta t<1/L$ imply that
\begin{align}\label{eqs3_29} \tilde b_j^l&\ge \frac{1}{\Delta t} \tilde{W}_j^{l-1}-L\tilde{W}_j^{l-1}  >0.
\end{align}
The combination of \eqref{eqs3_27} and \eqref{eqs3_29} immediately gives $\tilde{W}_j^{l}>0$ for $1\le j\le M-1$. Therefore, by mathematical induction, for $1\le n\le N$, if $\Delta t<1/L$, then the inequalities in \eqref{eqs3_18} and \eqref{eqs3_19} hold.
\end{prf}
\begin{thm}
For numerical scheme \eqref{eqs2_14}, if $\Delta t<1/L$, then the numerical solution satisfies
\begin{equation}\label{eqs3_30}
\hat W_j^n>0,\ 1\leq n\leq N,\  0\le j\le M-1,
\end{equation}
and
\begin{equation}\label{eqs3_31}
\hat R^{n}>\hat R^{n-1},\  1\le n\le N.
\end{equation}
\end{thm}
\begin{proof}
First, \eqref{eqs2_14.a} implies that
\begin{align*}
&\left( \frac{\hat R^n \hat z_j}{\Delta t} + \frac{2d \hat z_j}{(\Delta \hat z)^2} + \frac{d}{\Delta \hat z} + \frac{\hat z_j^2}{2} \frac{\hat R^n - \hat R^{n-1}}{\Delta t\Delta \hat z} \right) \hat W_j^n
- \left( \frac{d \hat z_j}{(\Delta \hat z)^2} + \frac{d}{\Delta \hat z} + \frac{\hat z_j^2}{2} \frac{\hat R^n - \hat R^{n-1}}{\Delta t\Delta \hat z} \right) \hat W_{j+1}^n
\notag\\
&- \frac{d \hat z_j}{(\Delta \hat z)^2}\hat  W_{j-1}^n = \frac{\hat R^n \hat z_j}{\Delta t} \hat W_j^{n-1} + R^n \hat z_j f(\hat W_j^{n-1}), \quad 1 \leq j \leq M-1, \ 1\le n\le N.
\end{align*}
By \eqref{eqs2_14.b}, we obtain $\hat W_0^n=\hat W_1^n$. Thus, the above equation can be rewritten as
\begin{equation}\label{eqs3_32}
    \hat A_n\hat {\mathbf W}_n=\hat{\mathbf b}_n,
\end{equation}
where
$
\hat {\mathbf W}_n=(\hat W_1^n,\hat W_2^n,\cdots ,\hat W_{M-1})^{\top},
$
$\hat{\mathbf b}_n=(\hat{b}_1^n, \hat{b}_2^n, \ldots, \hat{b}_{M-1}^n)^\top$ with $\hat{b}_j^n=\frac{\hat R^n \hat z_j}{\Delta t} \hat W_j^{n-1} + \hat R^n \hat z_j f(\hat W_j^{n-1}),$
and the matrix $\hat A_n=(\hat{a}_{ij}^n)_{(M-1)\times (M-1)}$ is tridiagonal with
\begin{equation}\label{eqs3_33}
\begin{cases}
\hat a_{11}^n=\frac{\hat R^n \hat z_j}{\Delta t} + \frac{d \hat z_j}{(\Delta \hat z)^2} + \frac{d}{\Delta \hat z} + \frac{\hat z_j^2}{2} \frac{\hat R^n - \hat R^{n-1}}{\Delta t\Delta \hat z}, \\
\hat a_{ij}^n=\frac{\hat R^n \hat z_j}{\Delta t} + \frac{2d \hat z_j}{(\Delta \hat z)^2} + \frac{d}{\Delta \hat z} + \frac{\hat z_j^2}{2} \frac{\hat R^n - \hat R^{n-1}}{\Delta t\Delta \hat z}, \  i=j,\ i\ge 2,\\
\hat a_{ij}^n=- \left( \frac{d \hat z_j}{(\Delta \hat z)^2} + \frac{d}{\Delta \hat z} + \frac{\hat z_j^2}{2} \frac{\hat R^n - \hat R^{n-1}}{\Delta t\Delta \hat z} \right) , \ i=j-1,\\
\hat a_{ij}^n=- \frac{d \hat z_j}{(\Delta \hat z)^2}, \  i=j+1,\\
\hat a_{ij}^n=0,\ |i-j|\ge 2.
\end{cases}
\end{equation}
For the case $n=1$, \eqref{eqs2_14.b}, \eqref{eqs2_14.c} and $\hat W_j^0= u_0(\hat z_j h_0)>0$ yield
\begin{equation}\label{eqs3_34}
\hat R^1=\hat R^0+\frac{2\mu \Delta t}{\Delta \hat z}\hat W_{M-1}^{0}>\hat R^0.
\end{equation}
From \eqref{eqs3_33}, \eqref{eqs3_34}, $\Delta t<1/L$ and \eqref{eqs1_4}, we obtain that the matrix $\hat{A}_1$ in \eqref{eqs3_32} satisfies
\begin{equation}\label{eqs3_35}
     \hat a_{ii}^1 > \sum_{j\neq i} |\hat a_{ij}^1|,\ \hat a_{ij}^1<0,\ |i-j|=1,
\end{equation}
and
\begin{align}\label{eqs3_37}
\hat b_j^1 &\ge \frac{\hat R^1 \hat z_j}{\Delta t} \hat W_j^{0}-L{\hat R^1 \hat z_j} \hat W_j^{0}
>0.
\end{align}
By combining \eqref{eqs3_35} and \eqref{eqs3_37} with Lemma \ref{lem3}, we can obtain $\hat W_j^1>0$ for $1\le j\le M-1$.
From \eqref{eqs2_14.b}, one further has $\hat W_0^1=\hat W_1^1>0$. Therefore, when $\Delta t<\frac{1}{L}$, we can conclude that \eqref{eqs3_30} and \eqref{eqs3_31} hold for $n=1$.

Assume that for a fixed integer $l\ge 2$, \eqref{eqs3_30} and \eqref{eqs3_31} hold for all $1\le n\le l-1$ under $\Delta t<\frac{1}{L}$. We now consider the case $n=l$. In view of  \eqref{eqs2_14.b}, \eqref{eqs2_14.c} and the above assumption, it gives
\begin{equation}\label{eqs3_38}
\hat R^l=\hat R^{l-1}+\frac{2\mu \Delta t}{\Delta \hat z}\hat W_{M-1}^{l-1}>\hat R^{l-1}.
\end{equation}
Based on
\eqref{eqs3_33}, \eqref{eqs3_38}, $\Delta t < \frac{1}{L}$ and \eqref{eqs1_4}, the elements of the matrix $\hat{A}_l$ and the vector $\hat{\mathbf{b}}_l$ in \eqref{eqs3_32} satisfy
\begin{equation}\label{eqs3_39}
     \hat a_{ii}^l > \sum_{j\neq i} |\hat a_{ij}^l|,\ \hat a_{ij}^l<0,\ |i-j|=1,
\end{equation}
and
\begin{align}\label{eqs3_41}
\hat b_j^l &\ge \frac{\hat R^l \hat z_j}{\Delta t} \hat W_j^{l-1}-L{\hat R^l \hat z_j} \hat W_j^{l-1}>0.
\end{align}
According to \eqref{eqs3_39} and \eqref{eqs3_41}, it is obvious that $\hat W_j^l>0,\ 1 \le j \le M-1$. Furthermore, from \eqref{eqs2_14.b}, we obtain $\hat W_0^l=\hat W_1^l>0$. It means that $\hat W_j^l>0,\ 0 \le j \le M-1$.
Thus, by mathematical induction, for $1\le n \le N$, if $\Delta t <\frac{1}{L}$, then \eqref{eqs3_30} and \eqref{eqs3_31} hold.
\end{proof}
\subsection{Stability}
Now we discuss the stability of the numerical schemes \eqref{eqs2_12}--\eqref{eqs2_14}.
Let \( V_{1} = \{v_j \mid 0\le j \le M,\ v_M=0\} \).
In order to derive the stability of the numerical scheme \eqref{eqs2_12}, we first define a useful norm.
For \( u, v \in V_{1} \), define
$
\|u\| = \sqrt{(u, u)} \text{ with }(u, v) = \Delta z \left( \frac{1}{2} u_0 v_0 + \sum_{j=1}^{M-1} u_j v_j \right).\label{eqs3_42}
$
Based on the defined norm, we have the following result.
\begin{thm}
Let $\{W_j^n \mid 0 \leq j \leq M-1,\ 0 \leq n \leq N\}$ be the solution to the numerical scheme \eqref{eqs2_12}. If $\Delta t < \frac{1}{2L}$, then there exists a constant $C$ such that
$
\left\| W^n \right\| \leq C\|W^{0}\|.
$
\end{thm}
\begin{proof}
 Multiplying both sides of \eqref{eqs2_12.a} by $\Delta z W_0^n$, and multiplying both sides of \eqref{eqs2_12.b} by $\Delta z W_j^n$ and summing from $j=1$ to $M-1$, we get
    \begin{equation}\label{eqs3_44}
    R^n \delta_t W_0^n\cdot \Delta z W_0^n - d\delta_z^2W_0^n \cdot \Delta z W_0^n = R^nf(W_0^{n-1}) \Delta z W_0^n,
    \end{equation}
and
\begin{equation}\label{eqs3_45}
\sum_{j=1}^{M-1}R^n \delta_t W_j^n \cdot\Delta z W_j^n - \sum_{j=1}^{M-1} \frac{z_j\delta_t R^n}{2} \delta_z W_{j+\frac12}^n \cdot \Delta z W_j^n
- \sum_{j=1}^{M-1}d\delta_z^2 W_j^n \cdot \Delta z W_j^n
= \sum_{j=1}^{M-1}R^n{f} (W_j^{n-1}) \cdot\Delta z W_j^n .
\end{equation}
By adding \eqref{eqs3_44} and \eqref{eqs3_45} side by side, one obtains
\begin{align}\label{eqs3_46}
&\sum_{j=1}^{M-1} R^n\delta_t W_j^n \cdot\Delta z W_j^n + R^n \delta_t W_0^n\cdot \Delta z W_0^n - \sum_{j=1}^{M-1} \frac{z_j\delta_t R^n}{2} \delta_z W_{j+\frac12}^n \cdot \Delta z W_j^n\notag\\
&-\sum_{j=1}^{M-1}d\delta_z^2 W_j^n \cdot \Delta z W_j^n -  d\delta_z^2W_0^n \cdot \Delta z W_0^n
= \sum_{j=1}^{M-1}R^n{f} (W_j^{n-1}) \cdot\Delta z W_j^n  + R^nf(W_0^{n-1}) \cdot\Delta z W_0^n.
\end{align}
Since
\begin{align}\label{eqs3_47}
    \sum_{j=1}^{M-1}\delta_z^2 W_j^n \cdot W_j^n&=\sum_{j=1}^{M-1}\frac{1}{\Delta z}(\delta_z W_{j+\frac12}^n-\delta_z W_{j-\frac12}^n)\cdot W_j^n
    =\frac{1}{\Delta z}\sum_{j=1}^{M-1}\delta_z W_{j+\frac12}^n\cdot W_j^n-\frac{1}{\Delta z}\sum_{j=0}^{M-2}\delta_z W_{j+\frac12}^n\cdot W_{j+1}^n,
\end{align}
and in view of \eqref{eqs3_44}, it follows that
\begin{equation}\label{eqs3_48}
      d\delta_z^2W_0^n \cdot \Delta z W_0^n = R^n \delta_t W_0^n\cdot \Delta z W_0^n-R^nf(W_0^{n-1}) \Delta z W_0^n.
\end{equation}
Then, using \eqref{eqs3_47} and \eqref{eqs3_48} to replace the fourth and fifth terms on the left-hand side of \eqref{eqs3_46} yields
\begin{align}\label{eqs3_49}
&\sum_{j=1}^{M-1} R^n\delta_t W_j^n \cdot\Delta z W_j^n + \frac{R^n}{2}\delta_t W_0^n\cdot \Delta z W_0^n - \sum_{j=1}^{M-1} \frac{z_j\delta_t R^n}{2} \delta_z W_{j+\frac12}^n \cdot \Delta z W_j^n
+ d \sum_{j=0}^{M-1} \Delta z \cdot \left(W_{j+\frac12}^n \right)^2\notag\\
= &\sum_{j=1}^{M-1}R^n{f} (W_j^{n-1}) \cdot\Delta z W_j^n + \frac{1}{2}R^nf(W_0^{n-1}) \Delta z W_0^n.
\end{align}
Note that
\begin{equation}\label{eqs3_n}
    (W_j^n-W_j^{n-1})W_j^n= \frac{1}{2} \left[(W_j^n)^2 - (W_j^{n-1})^2\right] + \frac{1}{2}(W_j^n - W_j^{n-1})^2 ,
\end{equation}
and
\begin{equation}\label{eqs3_j}
    (W_j^n-W_{j+1}^{n})W_j^n= \frac{1}{2} \left[(W_j^n)^2 - (W_{j+1}^{n})^2\right] + \frac{1}{2}(W_j^n - W_{j+1}^{n})^2.
\end{equation}
Therefore, the first, second and third terms on the left-hand side of \eqref{eqs3_49} can be bounded by
\begin{align}\label{eqs3_52}
    &\sum_{j=1}^{M-1} R^n\delta_t W_j^n \cdot\Delta z W_j^n + \frac{R^n}{2}\delta_t W_0^n\cdot \Delta z W_0^n - \sum_{j=1}^{M-1} \frac{z_j\delta_t R^n}{2} \delta_z W_{j+\frac12}^n \cdot \Delta z W_j^n \notag\\
    \ge &\sum_{j=1}^{M-1}\frac{\Delta zR^n}{2\Delta t}\left[(W_j^n)^2+\frac{1}{2}(W_0^n)^2-(W_j^{n-1})^2-\frac{1}{2}(W_0^{n-1})^2\right]+\sum_{j=1}^{M-1}\frac{\Delta z\delta_t R^n}{4}\left[j(W_j^n)^2 - j(W_{j+1}^{n})^2 \right]\notag\\
    \ge &\frac{R^n}{2\Delta t}\|W^n\|-\frac{R^n}{2\Delta t}\|W^{n-1}\|.
\end{align}
From \eqref{eqs1_4}, the right-hand side of \eqref{eqs3_49} can be estimated as
\begin{align*}
    &\sum_{j=1}^{M-1} R^n f(W_j^{n-1}) \cdot \Delta z W_j^n + \frac{1}{2}R^n f(W_0^{n-1}) \cdot \Delta z W_0^n\notag\\
    \le &\sum_{j=1}^{M-1} LR^n W_j^{n-1} \cdot \Delta z W_j^n + \frac{1}{2} LR^n W_0^{n-1}\cdot \Delta z W_0^n
    \le LR^n(W^{n-1},W^n).
\end{align*}
 By the Cauchy-Schwarz inequality, it yields
\begin{align}\label{eqs3_54}
    \sum_{j=1}^{M-1} R^n f(W_j^{n-1}) \cdot \Delta z W_j^n + \frac{1}{2}R^n f(W_0^{n-1}) \cdot \Delta z W_0^n
    \le \frac{1}{2}LR^n \| W^{n-1} \|^2 + \frac{1}{2}LR^n \| W^n \|^2.
\end{align}
Hence, it follows from \eqref{eqs3_49}--\eqref{eqs3_54} that
\begin{equation*}
    \|W^n\|^2\le \frac{1+L\Delta t}{1-L\Delta t}\|W^{n-1}\|^2,\ \text{i.e.,}\ \|W^n\|^2\le \left(1+\frac{2L\Delta t}{1-L\Delta t}\right)^n\|W^{0}\|^2.
\end{equation*}
%
Noting that $\lim_{y\to 0}(1+y)^{\frac{1}{y}}=e$ and$n\Delta t\le T$, we conclude that there exists a constant $C$ such that
$
\|W^n\|\le C\|W^0\|.
$
\end{proof}

Next, we employ an analogous argument to analyze the stability of numerical schemes \eqref{eqs2_13} and \eqref{eqs2_14}. In order to consider the numerical scheme \eqref{eqs2_13}, we introduce the following norm.
Let \( V_{2} = \{v_j \mid 0\le j \le M,\ v_0=0, \ v_{M}=0\} \). For \( u, v \in V_{2} \), denote
$
\|u\|_d = \sqrt{(u, u)_d} \text{ with } (u, v)_d = \Delta \tilde z \sum_{j=1}^{M-1} u_j v_j.
$
Then the following result hold.
\begin{thm}
Suppose that $\{\tilde W_j^n \mid 1 \leq j \leq M-1,\ 1 \leq n \leq N\}$ is the solution to the numerical scheme \eqref{eqs2_13}. If $\Delta t<\frac{1}{2L}$, then there exists a constant $C$ such that
$
\| \tilde W^n \|_d \leq C\|\tilde W^{0}\|_d.
$
\end{thm}
\begin{proof}
Multiplying both sides of \eqref{eqs2_13.a} by $\Delta \tilde z\tilde W_j^n$ and summing from $j=1$ to $M-1$ yields
\begin{align}\label{eqs3_60}
    &\sum_{j=1}^{M-1}\delta_t \tilde W_j^n\cdot \Delta \tilde{z} \tilde W_j^n
- \sum_{j=1}^{M-1}\dfrac{(1-\tilde{z}_j)\delta_t G^n}{H^n - G^n}\delta_{\tilde{z}} \tilde W_{j-\frac12}^n\cdot \Delta \tilde{z} \tilde W_j^n\notag\\
&- \sum_{j=1}^{M-1}\dfrac{\tilde{z}_j\delta_t H^n}{H^n - G^n}\delta_{\tilde{z}} \tilde W_{j+\frac12}^n\cdot \Delta \tilde{z} \tilde W_j^n
- \sum_{j=1}^{M-1}\frac{d}{(H^n - G^n)^2}\,\delta_{\tilde{z}}^2 \tilde W_j^n\cdot \Delta \tilde{z} \tilde W_j^n
= \sum_{j=1}^{M-1}f(\tilde W_j^{n-1})\cdot\Delta \tilde{z} \tilde W_j^n.
\end{align}
According to \eqref{eqs3_n} and \eqref{eqs3_j}, the first three terms on the left-hand side of \eqref{eqs3_60} gives
\begin{align}\label{eqs3_61}
    &\sum_{j=1}^{M-1}\delta_t \tilde W_j^n\cdot \Delta \tilde{z} \tilde W_j^n
- \sum_{j=1}^{M-1}\frac{(1-\tilde{z}_j)\delta_t  G^n}{H^n - G^n}\delta_{\tilde{z}} \tilde W_{j-\frac12}^n\cdot \Delta \tilde{z} \tilde W_j^n
- \sum_{j=1}^{M-1}\frac{\tilde{z}_j\delta_t H^n}{ H^n - G^n}\delta_{\tilde{z}} \tilde W_{j+\frac12}^n\cdot \Delta \tilde{z} \tilde W_j^n\notag\\
\ge &\sum_{j=1}^{M-1}\frac{\Delta \tilde z}{2\Delta t}\left[(\tilde W_j^n)^2-(\tilde W_j^{n-1})^2\right]- \sum_{j=1}^{M-1}\dfrac{\Delta \tilde{z}\delta_t G^n}{2(H^n -  G^n)}\left[(M-j)(\tilde W_j^n)^2-(M-j)(\tilde W_{j-1}^n)^2\right]\notag\\
&\ +\sum_{j=1}^{M-1}\dfrac{\Delta \tilde{z}\delta_t H^n}{2(H^n - G^n)}\left[j(\tilde W_j^n)^2-j(\tilde W_{j+1}^n)^2\right]\notag\\
\ge & \frac{1}{2\Delta t}\|\tilde W^n\|_d-\frac{1}{2\Delta t}\|\tilde W^{n-1}\|_d.
\end{align}
By using \eqref{eqs3_47}, the following estimate holds for the fourth term on the left-hand side of \eqref{eqs3_60}
\begin{align}\label{eqs3_62}
     - \sum_{j=1}^{M-1}\frac{d}{(H^n - G^n)^2}\,\delta_{\tilde{z}}^2 \tilde W_j^n\cdot \Delta \tilde{z} \tilde W_j^n
     \ge \sum_{j=1}^{M-1}\frac{d\Delta \tilde{z}}{(H^n - G^n)^2}(\delta_{\tilde{z}}\tilde{W})^2>0.
\end{align}
In addition, it follows from \eqref{eqs1_4} and the Cauchy-Schwarz inequality that
\begin{align}\label{eqs3_63}
\sum_{j=1}^{M-1}  f(\tilde W_j^{n-1}) \cdot \Delta \tilde{z}\tilde W_j^n
\leq \sum_{j=1}^{M-1} L \tilde W_j^{n-1} \cdot \Delta \tilde{z}  \tilde W_j^n
\leq L(\tilde W^{n-1}, \tilde W^n)\leq \frac{L}{2} \| \tilde W^{n-1} \|_d^2 + \frac{L}{2} \| \tilde W^n \|_d^2.
\end{align}
Therefore, the combination of \eqref{eqs3_61}, \eqref{eqs3_62} and \eqref{eqs3_63} immediately yields
\begin{equation*}
    \|\tilde W^n\|_d^2\le \frac{1+L\Delta t}{1-L\Delta t}\|\tilde W^{n-1}\|_d^2\leq\left(1+\frac{2L\Delta t}{1-L\Delta t}\right)^n\|\tilde W^{0}\|_d^2.
\end{equation*}
It is clear that the above result implies that
there exists a constant $C>0$ such that
$
\|\tilde W^n\|_d\le C\|\tilde W^0\|_d.
$
\end{proof}

To discuss the numerical scheme \eqref{eqs2_14}, for $u,v\in V_1,$ we define
$
(u, v)_z = \Delta \hat z \sum_{j=1}^{M-1}\hat z_j u_j v_j, \text{ and } \|u\|_z = \sqrt{(u, u)_z}.
$
Now we state the following result.
\begin{thm}
Suppose that $\{\hat W_j^n \mid 0 \leq j \leq M-1,\ 1 \leq n \leq N\}$ is the solution to the numerical scheme \eqref{eqs2_14}. If $\Delta t<\frac{1}{2L}$, then there exists a positive constant $C$ such that
$
\| \hat W^n \|_z \leq C\|\hat W^{0}\|_z.
$
\end{thm}
\begin{proof}
Multiplying both sides of \eqref{eqs2_14.a} by $\Delta \hat z \hat W_j^n$ and summing from $j=1$ to $M-1$ yields
\begin{align}  \label{eqs3_68}
&\sum_{j=1}^{M-1} \hat R^n\hat z_j\delta_t \hat W_j^n \cdot \Delta \hat z \hat W_j^n - \sum_{j=1}^{M-1} d\hat z_j\delta_{\hat z}^2 \hat W_j^n \cdot \Delta \hat z \hat W_j^{n} \notag\\
&- \sum_{j=1}^{M-1}\left( d + \frac{\hat z_j^2 \delta_t\hat R^n}{2} \right)\delta_{\hat z} \hat W_{j+\frac12}^n\cdot \Delta \hat z \hat W_{j}^n = \sum_{j=1}^{M-1} \hat R^n\hat z_j f(\hat W_j^{n-1})\cdot \Delta \hat z\hat W_j^n.
\end{align}
Applying \eqref{eqs3_n} to the first term on the left-hand side of \eqref{eqs3_68}, we get
\begin{align}   \label{eqs3_69}
\sum_{j=1}^{M-1} \frac{\hat R^n\hat z_j}{\Delta t} \left( \hat W_j^n - \hat W_j^{n-1} \right) \cdot \Delta \hat z \hat W_j^n
\geq \frac{\Delta \hat z\hat R^n}{2\Delta t}\sum_{j=1}^{M-1}\hat z_j\left( \left(\hat W_j^n\right)^2 - \left(\hat W_j^{n-1}\right)^2 \right)
=\frac{\hat R^n}{2\Delta t}\|\hat W^n\|_z^2-\frac{\hat R^n}{2\Delta t}\|\hat W^{n-1}\|_z^2.
\end{align}
According to \eqref{eqs3_j}, the second and third terms on the left-hand side of \eqref{eqs3_68} are given by
\begin{align}\label{eqs3_70}
    - \sum_{j=1}^{M-1} d\hat z_j\delta_{\hat z}^2 \hat W_j^n \cdot \Delta \hat z \hat W_j^{n}
    = & \frac{d}{\Delta\hat{z}} \sum_{j=1}^{M-1}\hat z_j\left[ (\hat W_j^n-\hat W_{j+1}^n)\hat W_j^n+(\hat W_j^n-\hat W_{j-1}^n)\hat W_j^n\right]\notag\\
    \ge & \frac{d}{2} \sum_{j=1}^{M-1}\left[\left(j(\hat W_j^n)^2-j(\hat W_{j+1}^n)^2\right)+\left(j(\hat W_j^n)^2-j(\hat W_{j-1}^n)^2\right)\right]\notag \\
    =&-\frac{d}{2}(\hat W_0^n)^2+\frac{Md}{2}(\hat W_{M-1}^n)^2
    >-\frac{d}{2}(\hat W_0^n)^2,
\end{align}
and
\begin{align}\label{eqs3_71}
    - \sum_{j=1}^{M-1}\left( d + \frac{\hat z_j^2 \delta_t\hat R^n}{2} \right)\delta_{\hat z} \hat W_j^n \cdot \Delta \hat z \hat W_j^n
    =&\sum_{j=1}^{M-1}\left( d + \frac{\hat z_j^2 \delta_t\hat R^n}{2} \right)(\hat W_j^n-\hat W_{j+1}^n)\hat W_j^n\notag\\
    \ge &\sum_{j=1}^{M-1}\left( \frac{d}{2} + \frac{\hat z_j^2 \delta_t\hat R^n}{4} \right)\left[(\hat W_j^n)^2 -(\hat W_{j+1}^n)^2\right]\notag \\
    =&\frac{d}{2}(\hat W_1^n)^2+\frac{\Delta \hat z ^2 \delta_t\hat R^n}{4}\sum_{j=1}^{M-1}j^2\left[(\hat W_j^n)^2 -(\hat W_{j+1}^n)^2\right]
    \ge  \frac{d}{2}(\hat W_1^n)^2.
\end{align}
Therefore, it follows from \eqref{eqs2_14.b} that
\begin{align}\label{eqs3_72}
    - \sum_{j=1}^{M-1} d\hat z_j\delta_{\hat z}^2 \hat W_j^n \cdot \Delta \hat z \hat W_j^{n}  - \sum_{j=1}^{M-1}\left( d + \frac{\hat z_j^2 \delta_t\hat R^n}{2} \right)\delta_{\hat z} \hat W_{j+\frac12}^n\cdot \Delta \hat z \hat W_{j}^n>0.
\end{align}
For the right-hand side of \eqref{eqs3_68}, the Cauchy-Schwarz inequality implies that
\begin{align}  \label{eqs3_73}
\sum_{j=1}^{M-1} \hat R^n \hat z_j f(\hat W_j^{n-1})\cdot \Delta \hat z\hat W_j^n
&\leq \sum_{j=1}^{M-1} \hat R^n L \hat z_j (\hat W_j^{n-1}) \Delta \hat z \hat W_j^n\notag\\
&\leq \frac{\Delta \hat zL \hat R^n}{2} \sum_{j=1}^{M-1} \hat z_j  \left( (\hat W_j^{n-1})^2 + (\hat W_j^n)^2 \right)\notag\\
&= \frac{L \hat R^n}{2} \|\hat W^{n-1} \|_z^2 +\frac{L \hat R^n}{2} \|\hat W^{n} \|_z^2.
\end{align}
By combining \eqref{eqs3_69}--\eqref{eqs3_73}, we can  deduce that
$
\|\hat W^n\|_z^2\le \frac{1+L\Delta t}{1-L\Delta t}\|\hat W^{n-1}\|_z^2
\le \left(1+\frac{2L\Delta t}{1-L\Delta t}\right)^n\|\hat W^{n-1}\|_z^2.
$
Thus the above derivation implies that there exists a constant $C$ such that
$
\|\hat W^n\|_z\le C\|\hat W^0\|_z.
$
\end{proof}
\section{Numerical examples}

In this section, we present some numerical tests on the problems \eqref{eqs1_1}, \eqref{eqs1_2} and \eqref{eqs1_3}, aiming to demonstrate the effectiveness of the developed numerical schemes \eqref{eqs2_12}, \eqref{eqs2_13} and \eqref{eqs2_14} on and the correctness of the theory.

\textbf{Example 4.1} Consider the one-dimensional  diffusive logistic problem with free boundary by taking
$f(u)=u(a - b u)$ in the problem \eqref{eqs1_1}.

For Example 4.1, Ref. \cite{du2010spreading} has shown that spreading always occurs when $h_0 \ge L$, where $L = \frac{\pi}{2}\sqrt{\frac{d}{a}}$.
When $h_0 < L$, there exists a threshold $\mu^*$ depending on $u_0(x)$ such that the spreading also occurs if $\mu > \mu^*$, whereas the species eventually vanishes if $0 < \mu \le \mu^*$.
In the spreading case, for any given $\mu>0$, there exists a spreading speed $k_0(\mu)>0$ such that, for sufficiently large $t$, \(\frac{dh}{dt}\to k_0(\mu)\) and the free boundary grows linearly in time. Meanwhile, \(u(t,x)\) converges locally to the carrying capacity. In the vanishing case, the free boundary $h(t)$ is monotonically increasing but remains bounded above by $L$, and the population density converges uniformly to zero in the habitat domain.

Motivated by the mentioned theory in \cite{du2010spreading}, we use numerical scheme \eqref{eqs2_12} to solve Example 4.1 on different cases to observe these results. For Example 1.1, we choose  $(d,\mu,a,b,h_0)=(1.0,1.0,2,1,4.0)$ and take $u_0(x)=\cos(\pi x/8)$.
Since $h_0=4>L\approx 1.11$, it means that the spreading occurs. Fig.~\ref{fig:1} shows that   the population density converges locally to the carrying capacity $a/b$ over time, and that the free boundary grows linearly in time.
By contrast, the parameters are set to $(d,\mu,a,b,h_0)=(3.0,1.0,2,1,1.0)$ and $u_0(x)=\cos(\pi x/2)$. Since $h_0<L\approx 1.92$, it implies that vanishing occurs. Fig.~\ref{fig:2} indicates that the population density eventually decays to zero, and that the free boundary remains strictly bounded above by $L$ throughout the evolution.
For $(d,\mu,a,b,h_0)=(1.0,\mu,1.0,1,1.0)$ and $u_0(x)=\cos(\pi x/2)$, Fig.~\ref{fig:3} presents the evolution of free boundary  for different $\mu$.
Clearly, for the given parameters, there exists a unique threshold $\mu^*$
that separates spreading behavior from vanishing behavior, and $\mu^*$
can be estimated to lie  between 0.30 and 0.40. For $\mu>\mu^*$, the boundary grows linearly in time. The right panel plots $dh(t)/dt$, which tends to a positive constant in the spreading case and to zero in the vanishing case.
The numerical results in Figs.~\ref{fig:1}--\ref{fig:3} are consistent with the theoretical predictions in \cite{du2010spreading}. Besides, Figs.~\ref{fig:1} and~\ref{fig:2} confirm the positivity of the numerical population density, and Figs.~\ref{fig:1}--\ref{fig:3} demonstrate the monotonicity of the free boundary $h(t)$, both of which agree with our theoretical analysis.

\begin{figure}[htbp]
\centering
\includegraphics[width=0.45\columnwidth]{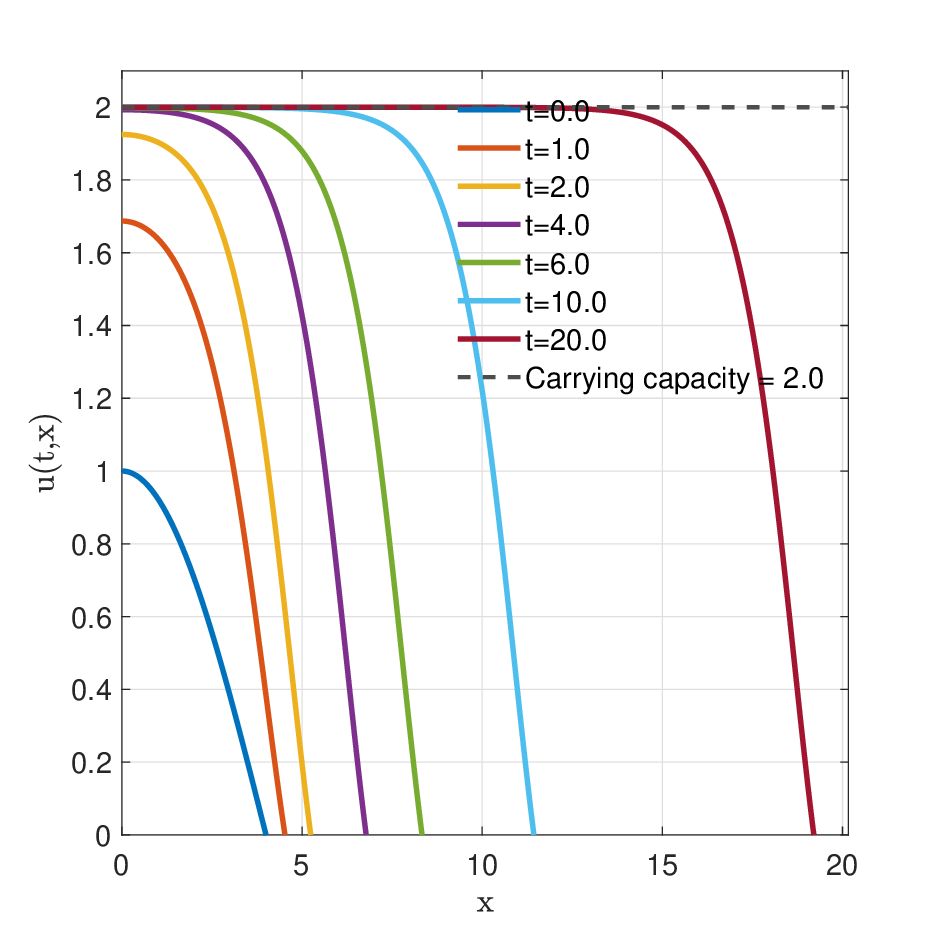}
\includegraphics[width=0.45\columnwidth]{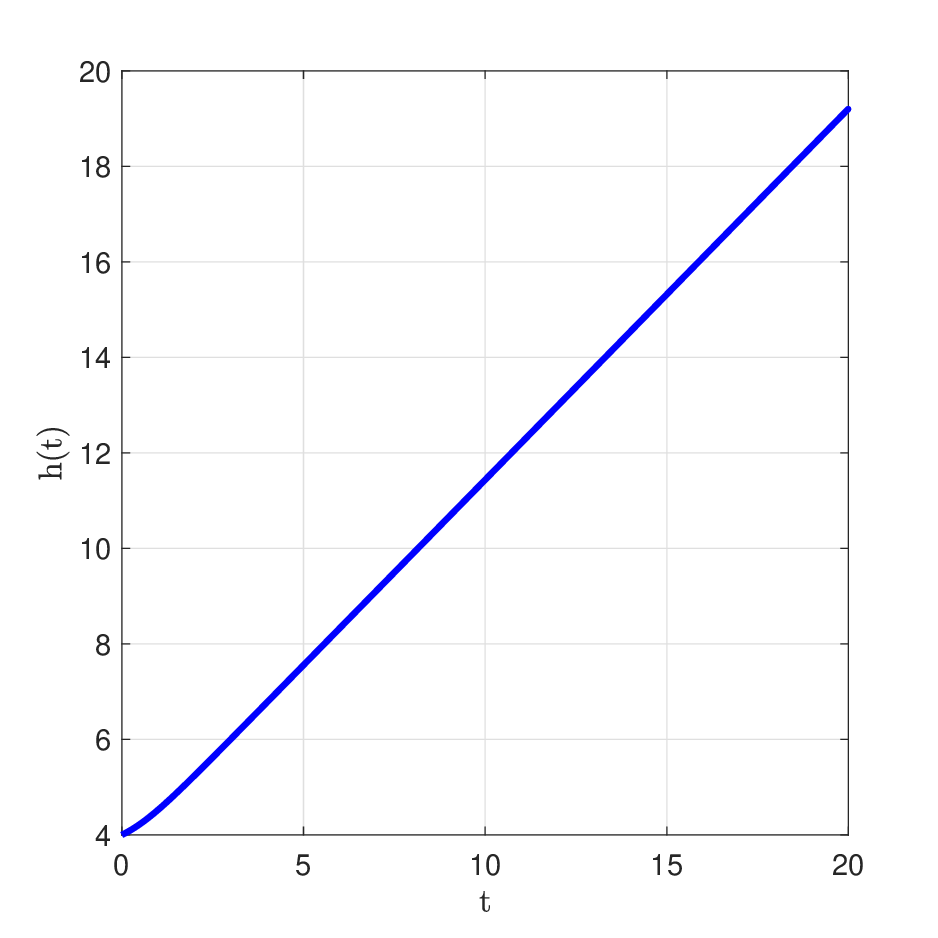}
\caption{Numerical solution of the population density $u(t,x)$ at different times (left), and numerical solution of the spreading front $h(t)$ under the same parameter values (right), where $\Delta t=\Delta z=0.001$.}
\label{fig:1}
\end{figure}

\begin{figure}[htbp]
\centering
\includegraphics[width=0.45\columnwidth]{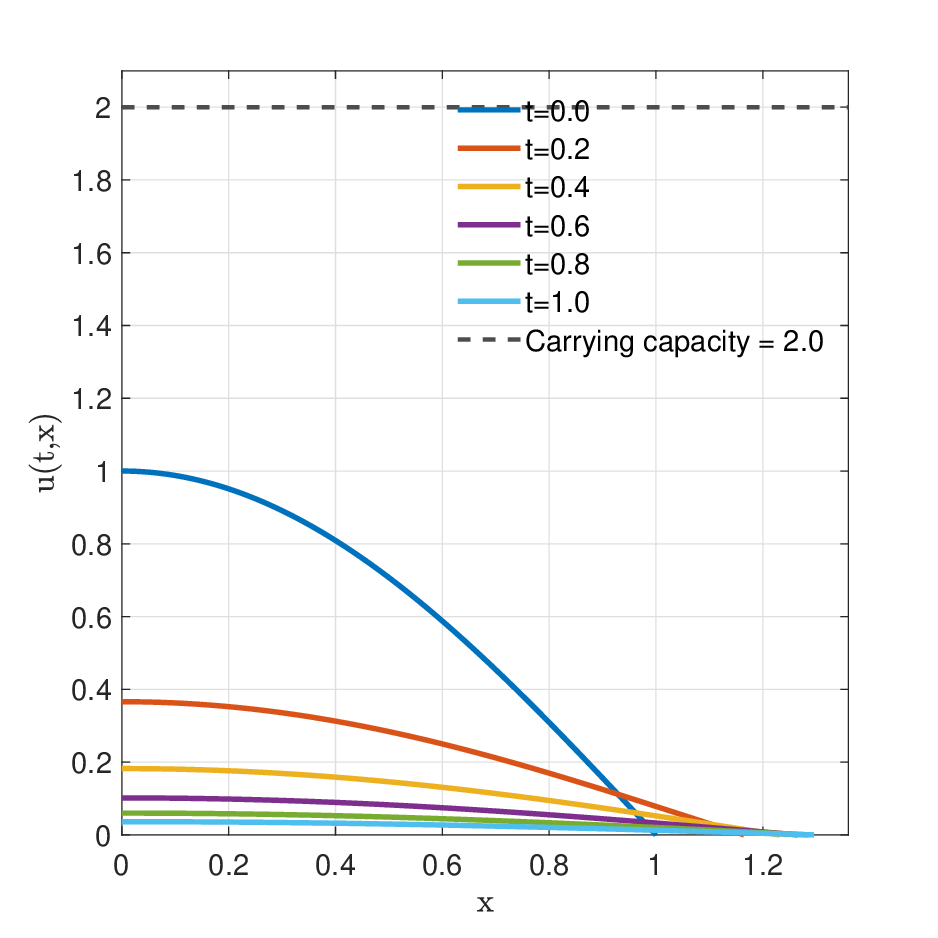}
\includegraphics[width=0.45\columnwidth]{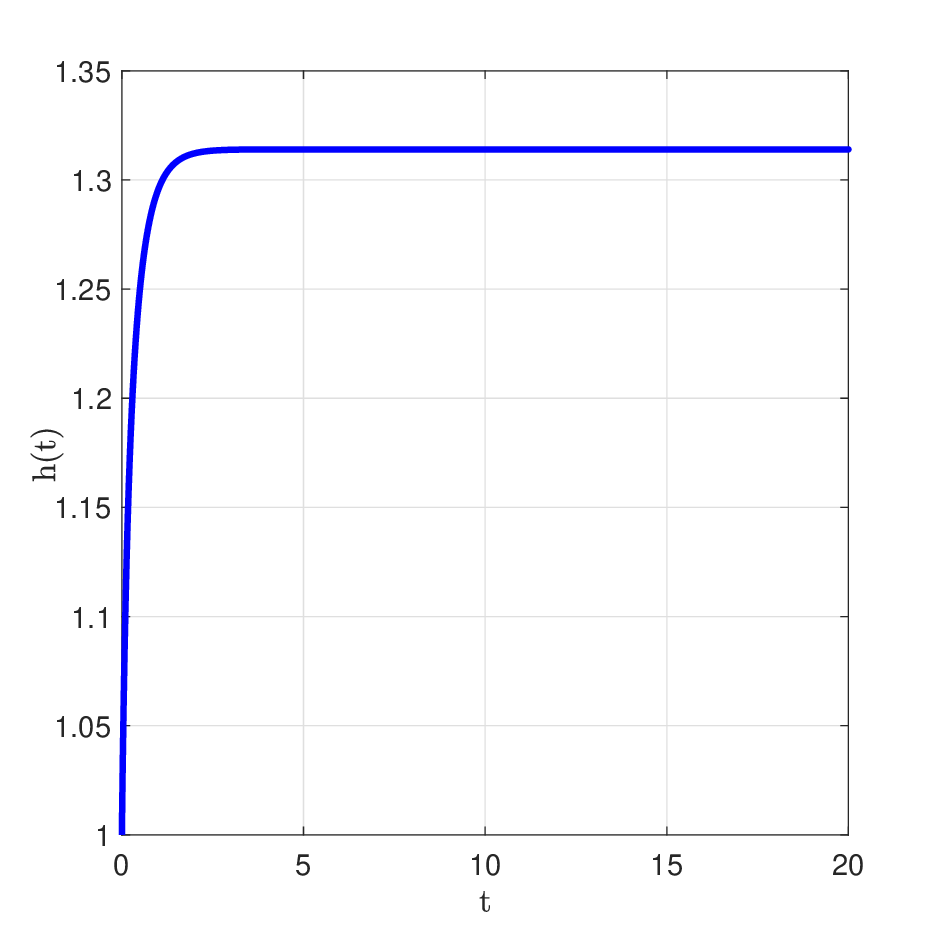}
\caption{Numerical solution of the population density $u(t,x)$ at different times (left), and Numerical solution of the spreading front $h(t)$ under the same parameter values (right), where $\Delta t=\Delta z=0.001$.}
\label{fig:2}
\end{figure}
\begin{figure}[htbp]
\centering
\includegraphics[width=0.45\columnwidth]{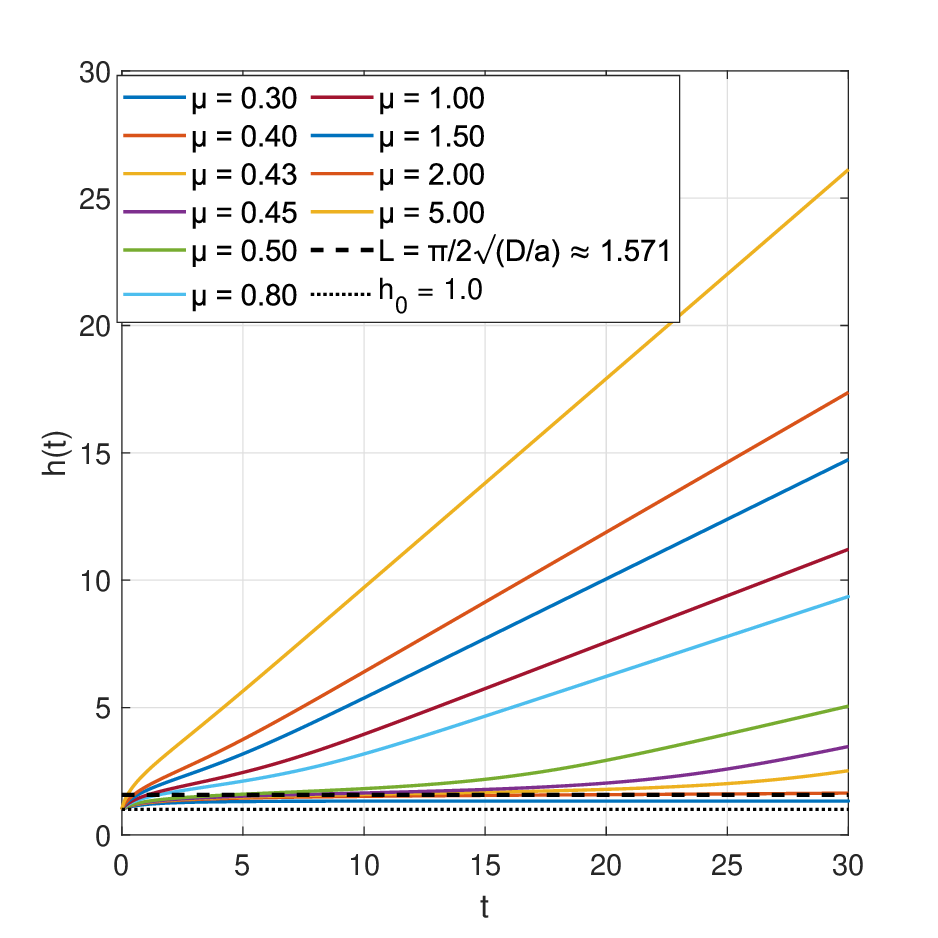}
\includegraphics[width=0.45\columnwidth]{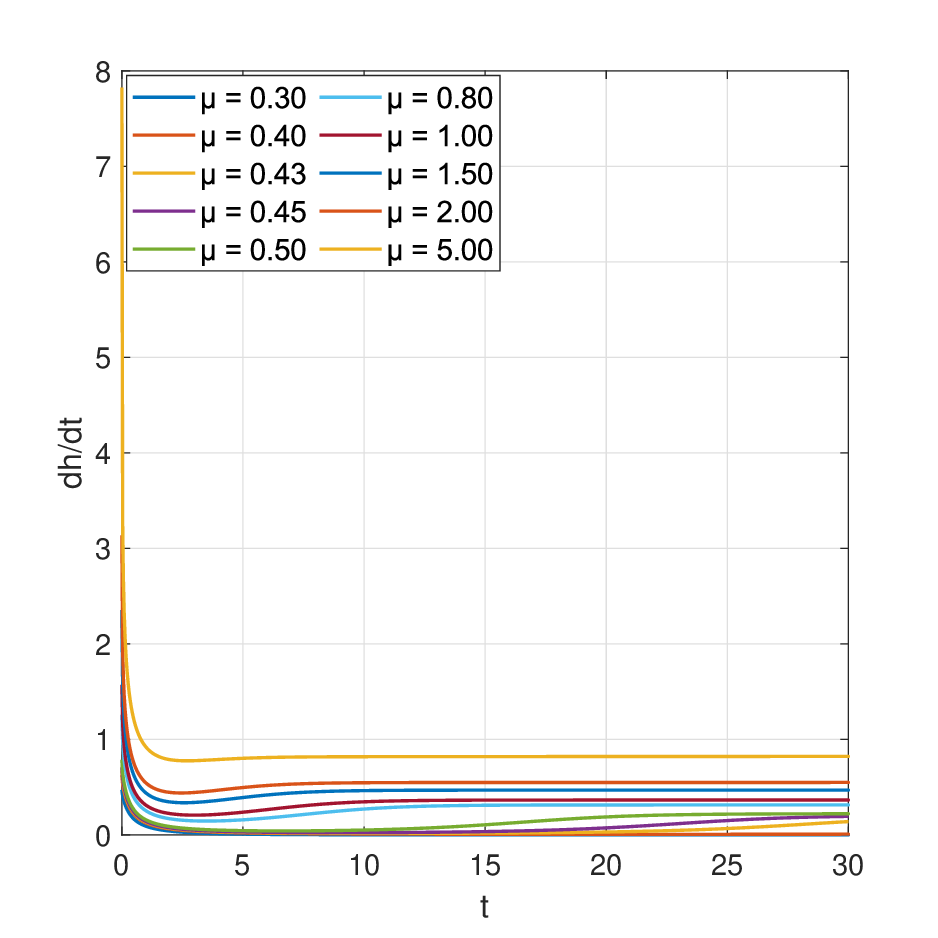}
\caption{Numerical solution of $h(t)$ for different values $\mu$ (left), and Numerical solution of $dh(t)/dt$ under the same parameter values (right), where $\Delta t=\Delta z=0.001$.}
\label{fig:3}
\end{figure}

\textbf{Example 4.2}
Consider the one-dimensional reaction-diffusion equation \eqref{eqs1_2} based on $\mu_1=\mu_2=\mu$ and nonlinear term $f(u)=u(1-u)(u-\theta)$,
$\theta\in (0,\frac{1}{2})$.

For Example 4.2 with $f(u)=u(1-u)(u-\theta)$, the theory in \cite{du2015spreading} indicates that, even when $h_0 \ge Z_B$, both the spreading and vanishing are still possible, where $F(u)=\int_0^u f(s)\,ds$, $\bar{\theta}=\{u\in(\theta,1):F(u)=0\}$, $Z_B=\inf_{\bar{\theta}<q<1}Z(q)$, and $Z(q)=\int_0^q dr/\sqrt{2(F(q)-F(r))}$. More precisely, if $h_0 \ge Z_B$, there exists a finite threshold $\sigma^* < \infty$ depending on $h_0$ and $u_0(x)$ such that the spreading occurs for $\sigma > \sigma^*$, whereas the vanishing occurs for $0 < \sigma < \sigma^*$.
In the spreading case, the boundary moves outward at a linear rate with a constant speed depending on
$\mu$, and the population density converges locally to the carrying capacity. In contrast, in the vanishing case, $g(t)$ and $h(t)$ tend to finite limits, whereas the population density converges uniformly to zero in the surviving region.

For Example 4.2 with $f(u)=u(1-u)(u-\theta)$, let $(d,\mu,h_0,\theta)=(1.0,1.0,5.0,0.2)$ and $u_0(x)=0.7\cos(\pi x/10)$. Since $h_0=5.0>Z_B\approx 4.13$, Fig.~\ref{fig:8} shows a spreading case in which the population density locally approaches the carrying capacity $1$ as time increases and the expanding fronts grow linearly.
With the same parameters but $u_0(x)=0.2\cos(\pi x/10)$, Fig.~\ref{fig:9} displays the vanishing case, in which the density tends uniformly to zero in the surviving region and both expanding fronts converge to finite limits.
For $(d,\mu,h_0,\theta)=(1.0,0.3,5.0,0.2)$ and $u_0(x)=\sigma\cos(\pi x/10)$, Fig.~\ref{fig:10} illustrates the evolution of the expanding fronts for different $\sigma$ when $h_0>Z_B\approx 4.13$.
For fixed $\mu$, there exists a unique threshold $\sigma^*$ between $0.40$ and $0.45$ separating the spreading and the vanishing. In the spreading case, the numerical derivatives $dh(t)/dt$ and $dg(t)/dt$ tend to a nonzero constant in time, whereas in the vanishing case they tend to zero. By observing the above these numerical results in in Figs.~\ref{fig:8}--\ref{fig:10}, it is clear that they are consistent with the theoretical findings in Ref. \cite{du2015spreading}.
These results in Figs.~\ref{fig:8}--\ref{fig:10} are consistent with the theoretical findings in Ref. \cite{du2015spreading}. Moreover, these obtained figures show that the numerical density remains positive, the numerical solution of $h(t)$ is monotonically increasing, and the numerical solution of $g(t)$ is monotonically decreasing, which matches our theory.
\begin{figure}[htbp]
\centering
\includegraphics[width=0.45\columnwidth]{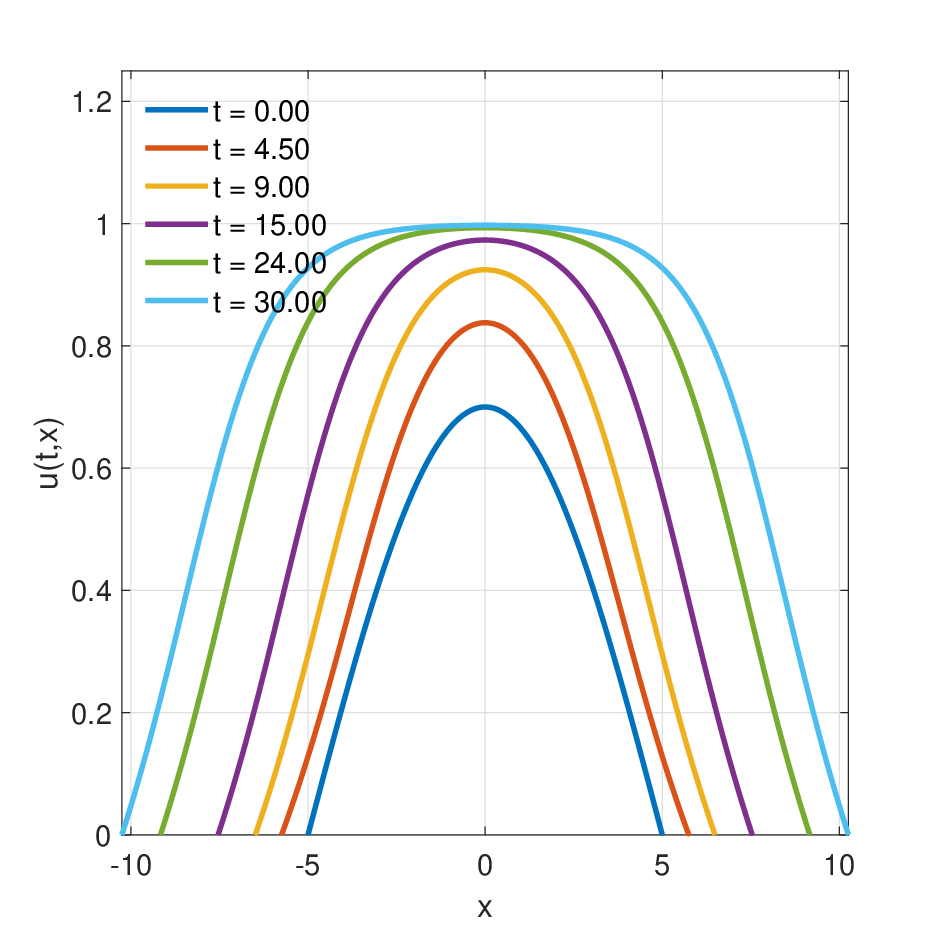}
\includegraphics[width=0.45\columnwidth]{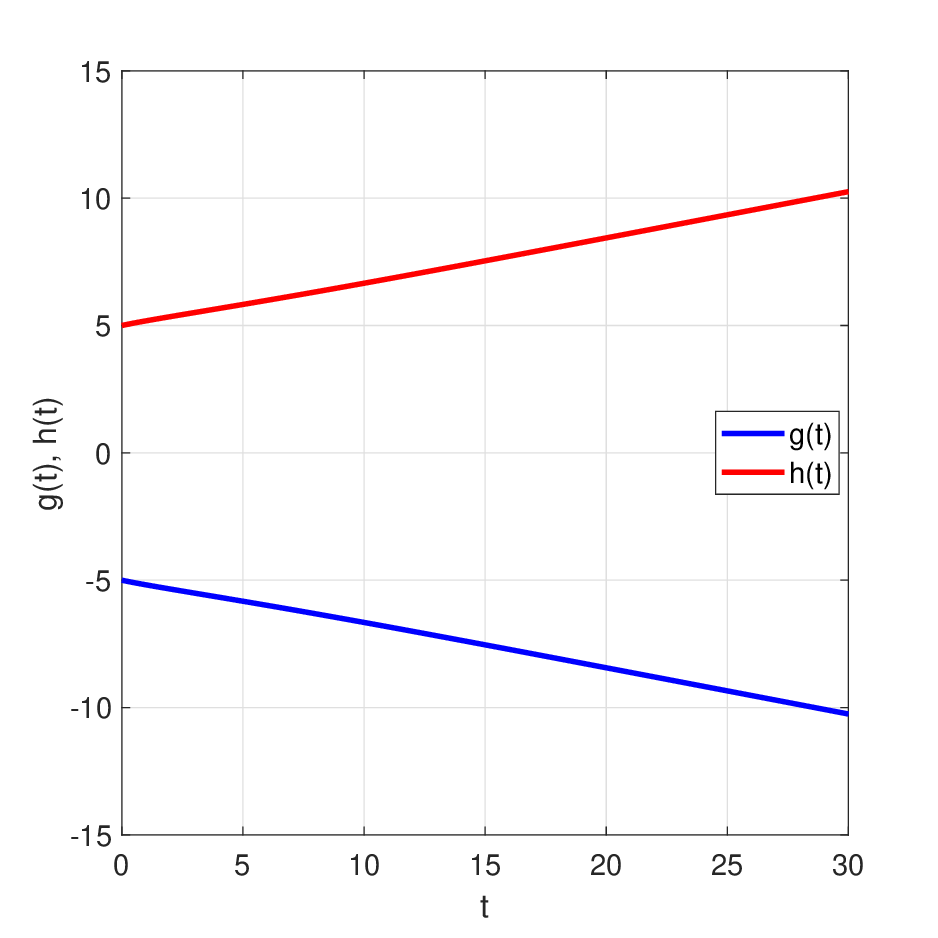}
\caption{Numerical solution of the population density $u(t,x)$ at different times (left), and numerical solution of the spreading front $h(t)$ under the same parameter values (right), where $\Delta t=\Delta \tilde z=0.001$.}
\label{fig:8}
\end{figure}
\begin{figure}[htbp]
\centering
\includegraphics[width=0.45\columnwidth]{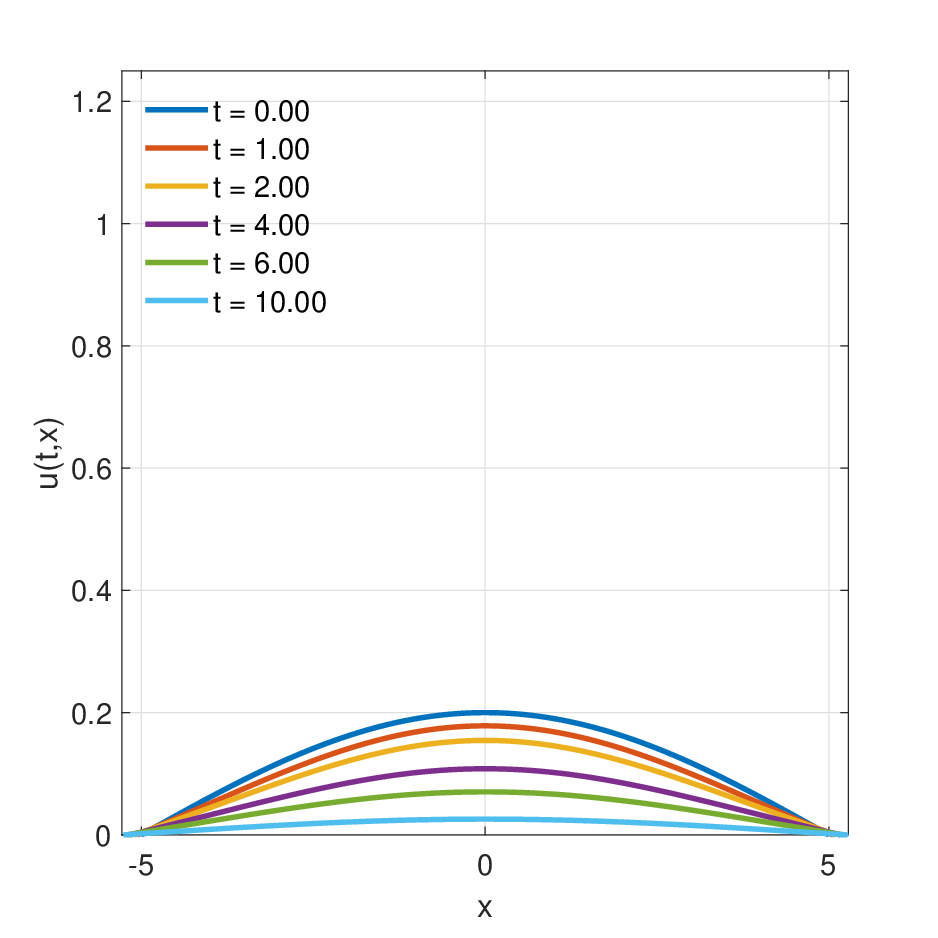}
\includegraphics[width=0.45\columnwidth]{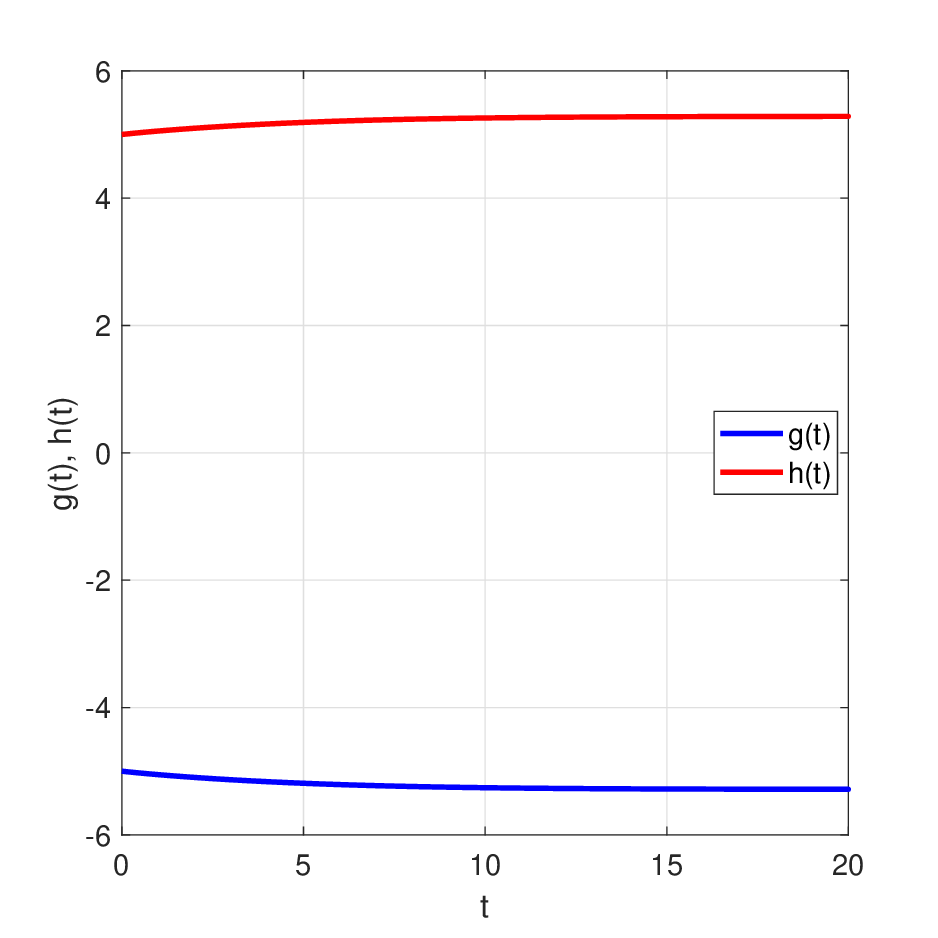}
\caption{Numerical solution of the population density $u(t,x)$ at different times (left), and numerical solution of the spreading front $h(t)$ under the same parameter values (right), where $\Delta t=\Delta \tilde z=0.001$.}
\label{fig:9}
\end{figure}
\begin{figure}[htbp]
\centering
\includegraphics[width=0.45\columnwidth]{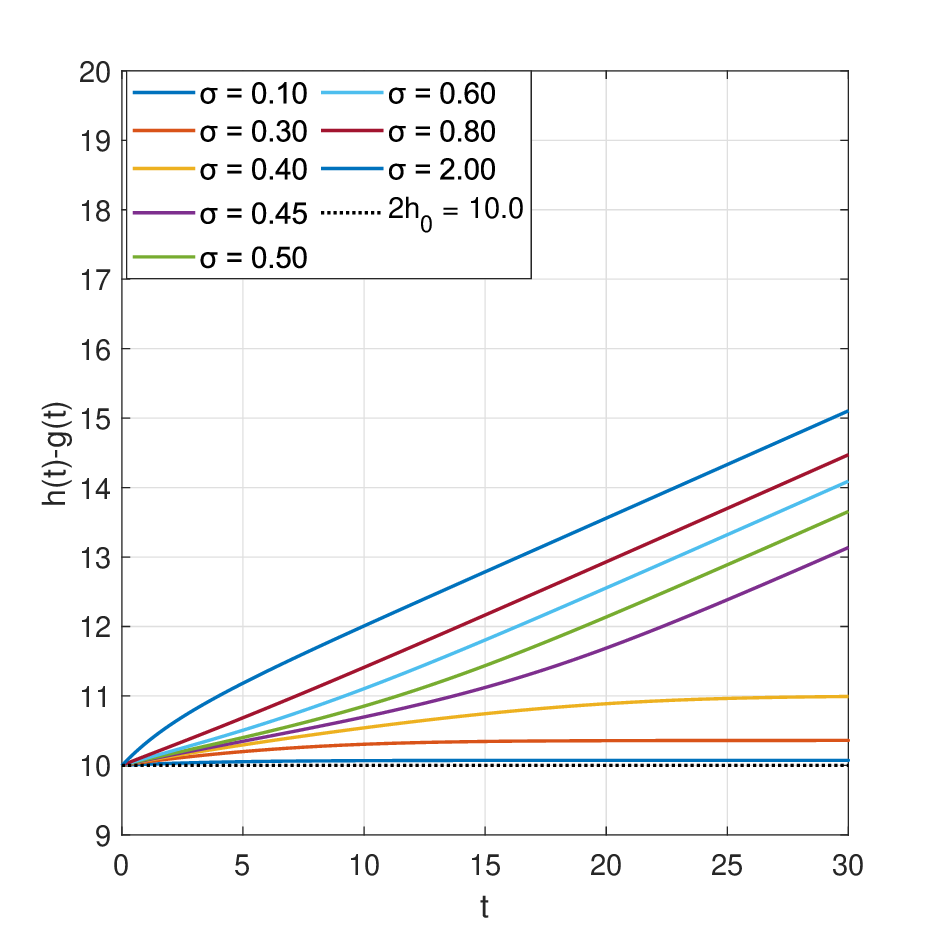}
\includegraphics[width=0.45\columnwidth]{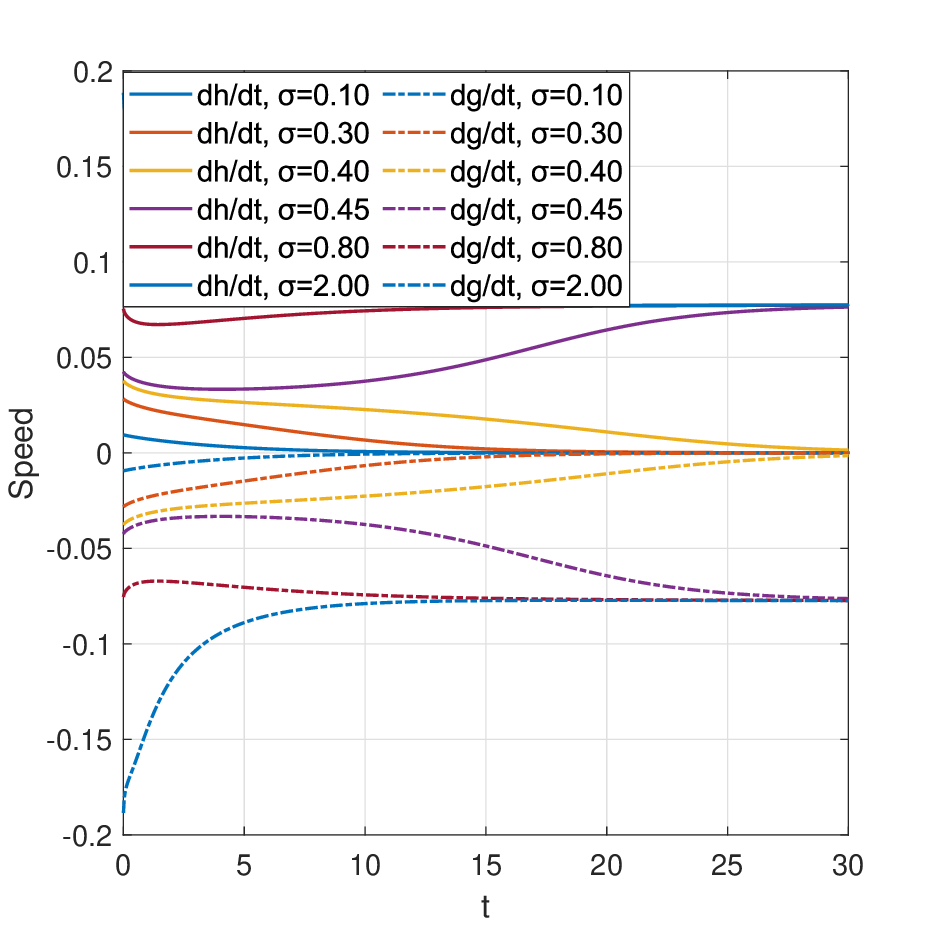}
\caption{Numerical solution of $h(x)-g(x)$ for different values $\sigma$ (left), and numerical solutions of $dh/dt$ and $dg/dt$ under the same parameter values(right), where $\Delta t=\Delta \tilde z=0.001$.}
\label{fig:10}
\end{figure}

\textbf{Example 4.3} Consider the reaction-diffusion equation \eqref{eqs1_3} with nonlinear term $f(u)=u(a-bu)$.

For Example 4.3 with $f(u)=u(a-bu)$, the results in \cite{du2011spreading} show that there exists $R^*\approx 2.4048\sqrt{\frac{d}{a}}$ such that the species spreads when $h_0\ge R^*$.Conversely, when $0<h_0<R^*$, there exists $\mu^*$ depending on $u_0(r)$ such that the species spreads if $\mu>\mu^*$, and the species eventually disappears if $0<\mu\le \mu^*$.Let $d=1, a=b$, the conclusions in \cite{Du&Matsuzawa2015Spreading} indicate that there is a threshold $\sigma^*$ depending on $u_0(r)$ and $h_0$ such that the species spreads if $\sigma>\sigma^*$, and the species eventually disappearsif $0<\sigma\le \sigma^*$. In the spreading case, there is $c^*$ depending on $\mu$ such that, for sufficiently large $t$, the spreading speed approaches $c^*$ and the population density $u(t,r)$ locally tends to the carrying capacity of the habitat. In the vanishing case, the spreading front $h(t)$ eventually converges to a finite value with $h(t)\le R^*$, and the population density $u(t,r)$ converges uniformly to $0$.

 \begin{figure}[htbp]
\centering
\includegraphics[width=0.45\columnwidth]{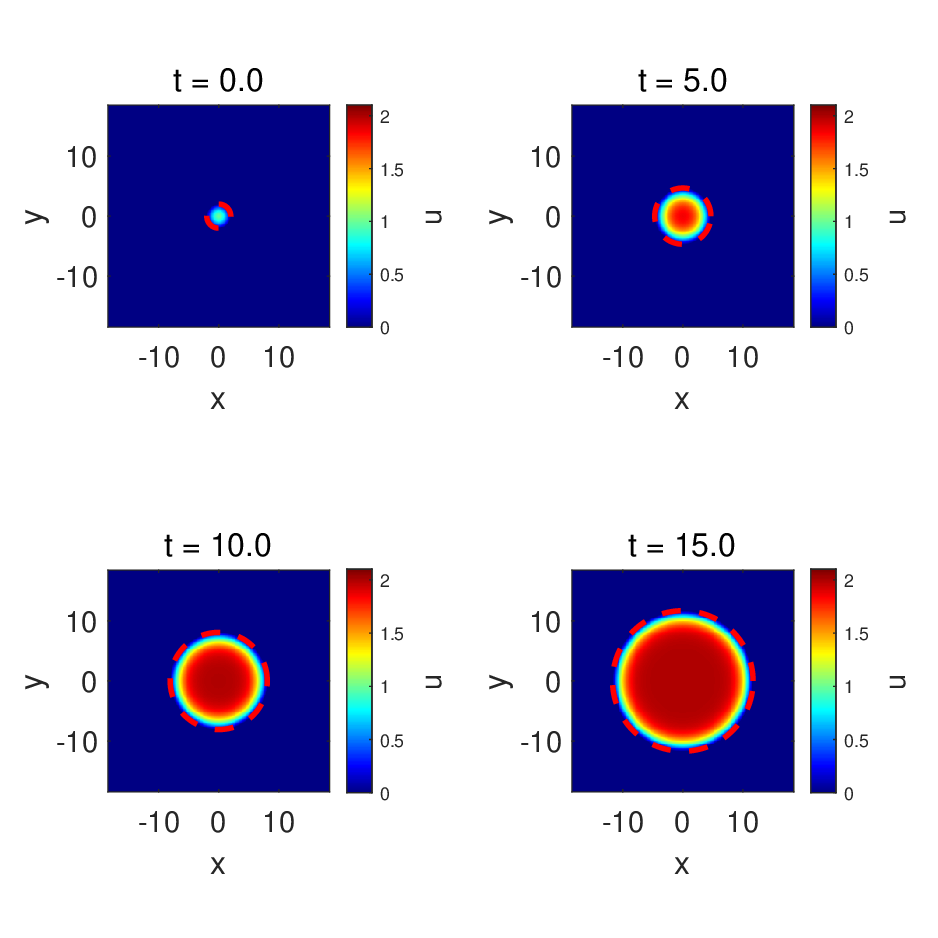}
\includegraphics[width=0.45\columnwidth]{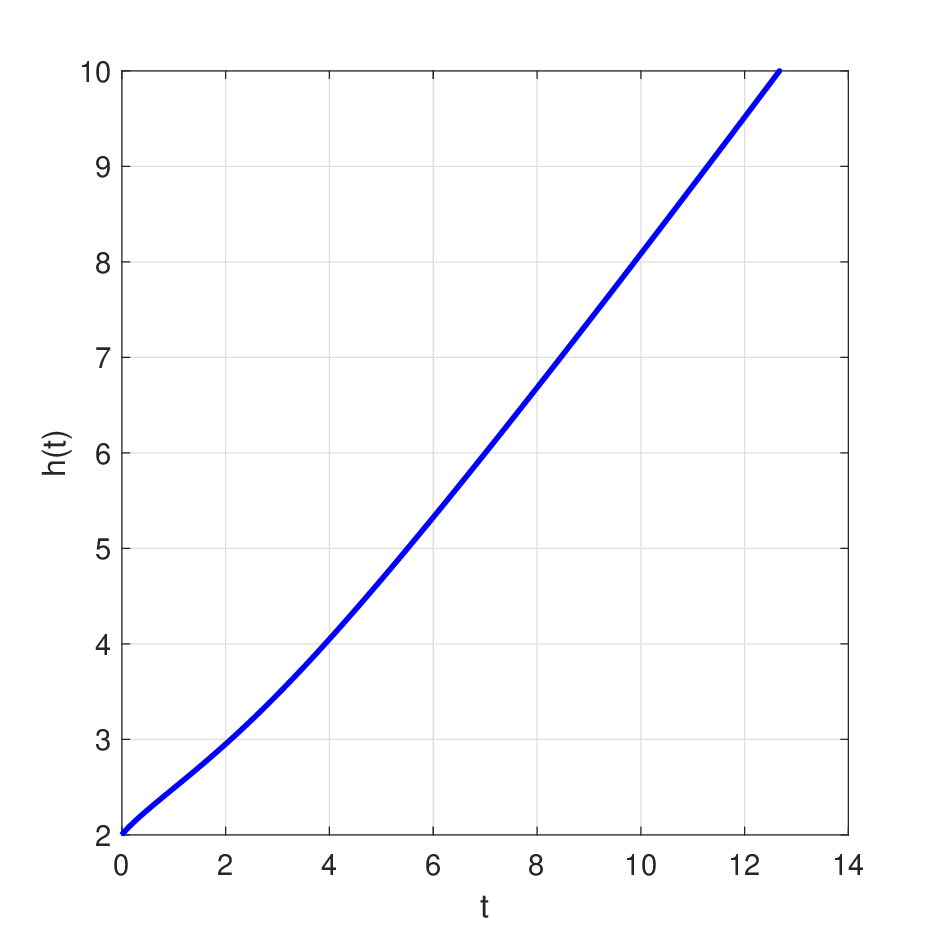}
\caption{Numerical solution of the population density $u(t,r)$ at different times (left), and numerical solution of the spreading front $h(t)$ under the same parameter values (right), where $\Delta t=\Delta\hat  z=0.001$.}
\label{fig:11}
\end{figure}
\begin{figure}[htbp]
\centering
\includegraphics[width=0.45\columnwidth]{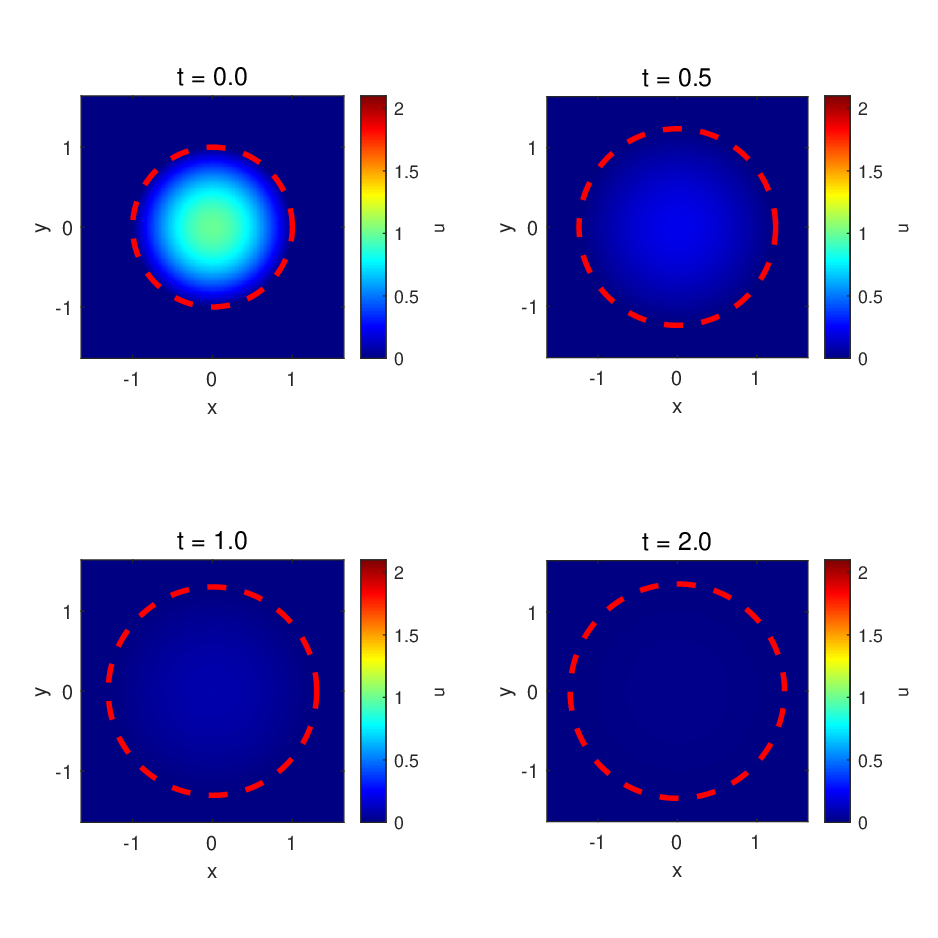}
\includegraphics[width=0.45\columnwidth]{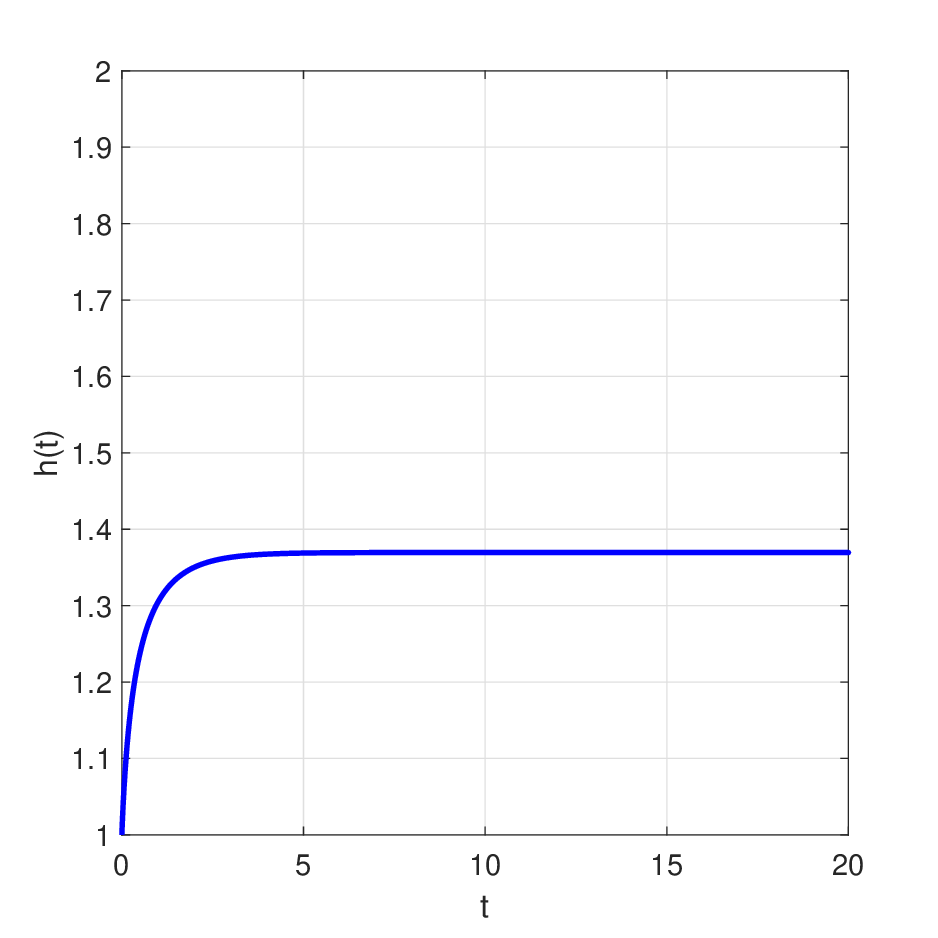}
\caption{Numerical solution of the population density $u(t,r)$ at different times (left), and numerical solution of the spreading front $h(t)$ under the same parameter values (right), where $\Delta t=\Delta \hat z=0.001$.}
\label{fig:12}
\end{figure}
\begin{figure}[htbp]
\centering
\includegraphics[width=0.45\columnwidth]{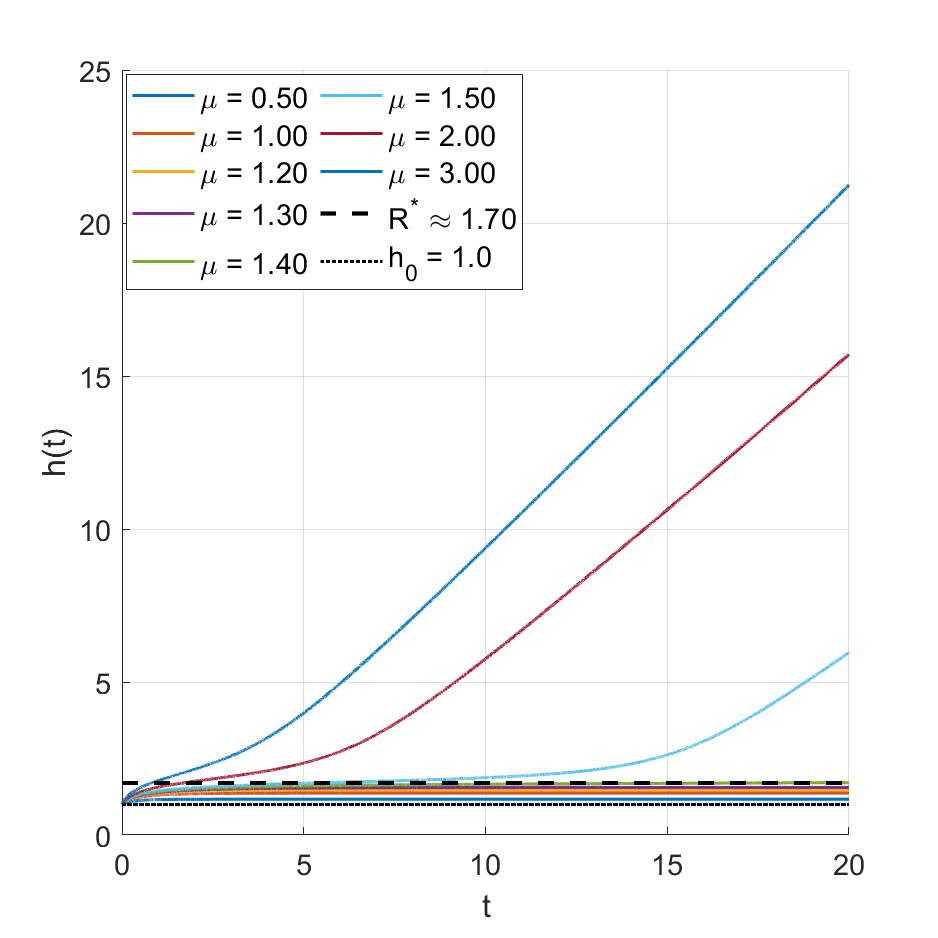}
\includegraphics[width=0.45\columnwidth]{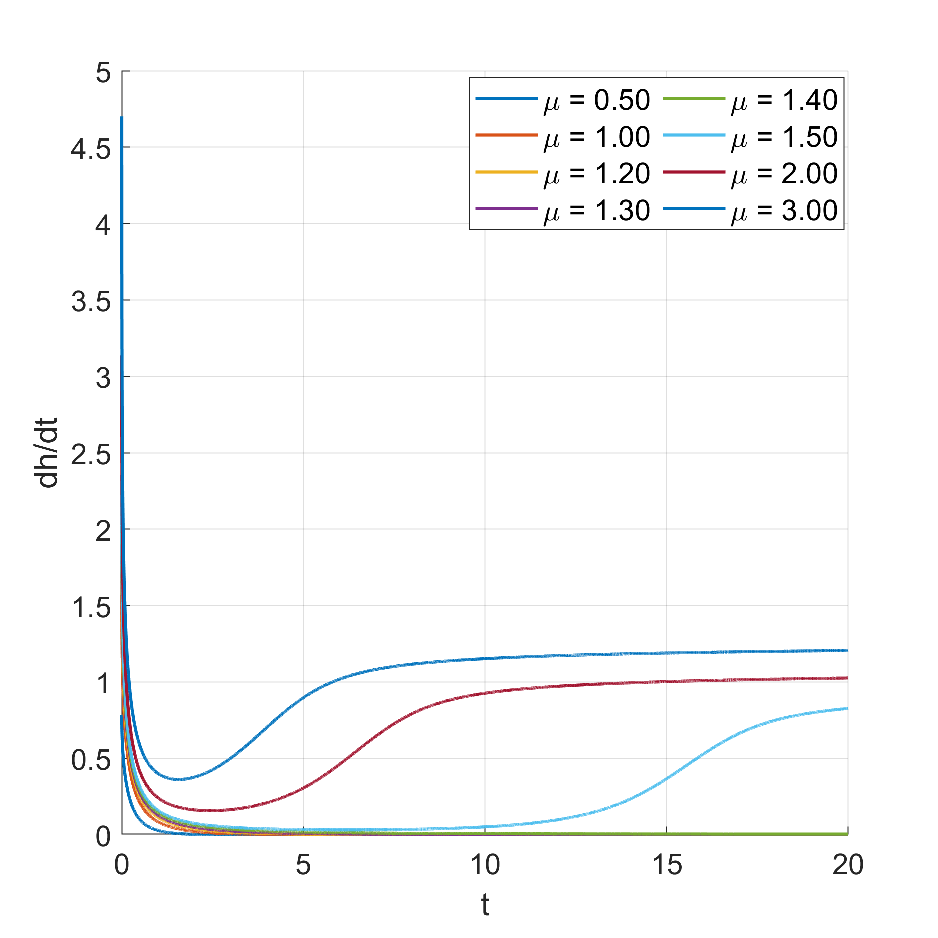}
\caption{Numerical solution of $h(t)$ for different  $\mu$ (left), and numerical solution of the derivative $dh(t)/dt$ under the same parameter values (right), where $\Delta t=\Delta \hat z=0.001$.}
\label{fig:13}
\end{figure}
\begin{figure}[htbp]
\centering
\includegraphics[width=0.45\columnwidth]{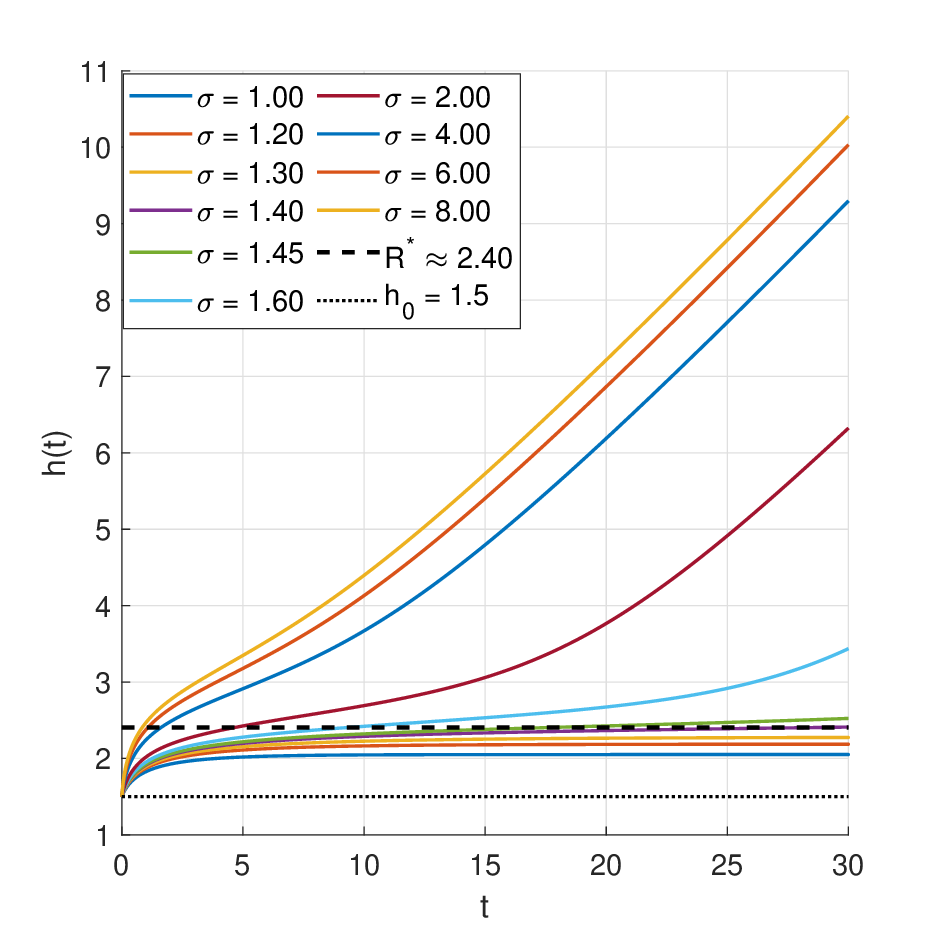}
\includegraphics[width=0.45\columnwidth]{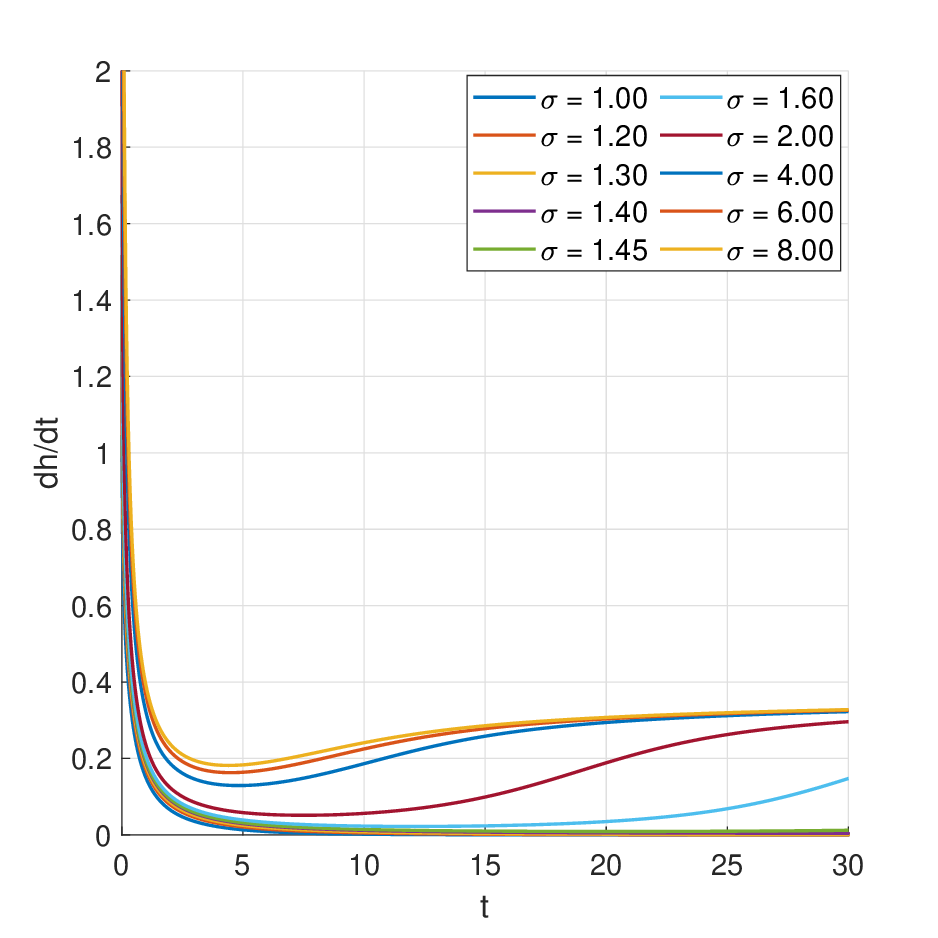}
\caption{Numerical solution of $h(t)$ for different  $\sigma$ (left), and numerical solution of the derivative $dh(t)/dt$ under the same parameter values (right), where $\Delta t=\Delta\hat z=0.001$.}
\label{fig:14}
\end{figure}

Now, we calculate Example 4.3 by numerical scheme \eqref{eqs2_14} and observe the obtained results. For Example 4.3 with $f(u)=u(a-bu)$, by taking parameters $(d,\mu,a,b,h_0)=(1.0,1.0,2,1,2.0)$ and $u_0(r)=\cos(\pi r/4)$, Fig.~\ref{fig:11} shows the spreading behavior of the species under the condition $h_0=2>R^*\approx 1.70$.
It indicates that the population density gradually approaches the habitat carrying capacity $a/b$ as time increases, and the spreading front gradually exhibits a linear expansion trend over time.
In contrast, let $(d,\mu,a,b,h_0)=(1.0,1.0,2,1,1.0)$ and $u_0(r)=\cos(\pi r/2)$. Fig.~\ref{fig:12} depicts the vanishing behavior as $h_0<R^*\approx 1.70$, where the population density gradually tends to $0$ as time increases and the spreading front remains bounded above by $R^*$.
Let $(d,\mu,a,b,h_0)=(1.0,\mu,1.0,1,1.0)$ and $u_0(r)=\cos(\pi r/2)$, Fig.~\ref{fig:13} represents that there exists a unique threshold $\mu^*\in(1.40,1.50)$ separating the spreading and the vanishing.
When $\mu>\mu^*$, the spreading front expands linearly in time. The figure also shows that the spreading speed approaches a positive constant in the spreading case, but tends to 0 in the vanishing case. Let $(d,\mu,a,b,h_0)=(1.0,1.0,1.0,1,1.5)$ and $u_0(r)=\sigma\cos(\pi r/3)$.
Fig.~\ref{fig:14} shows that there is a threshold $\sigma^*\in(1.45,1.60)$ separating the spreading and the vanishing. It also indicates that the spreading speed approaches a positive constant in the spreading case, but tends to 0 in the vanishing case.
The numerical results indicate that the spreading speed approaches a positive constant in the spreading case, whereas it tends to 0 in the vanishing case.
By observing the above Figs.~\ref{fig:11}--\ref{fig:14}, they are
obviously consistent with the theoretical results in Refs. \cite{du2011spreading,Du&Matsuzawa2015Spreading}.
What is more, these figures also clearly demonstrate that the numerical solution of the population density is remains positive, and the spreading front $h(t)$ is monotonically increasing, which is agreement with our theoretical results.

\section{Conclusion}
In order to solve the free boundary problems \eqref{eqs1_1}, \eqref{eqs1_2} and \eqref{eqs1_3}, we propose three efficient finite difference methods in this paper. Especially,
the popular front-fixing method is used to transform the considered equations into some problems
with fixed boundary firstly. Then the finite difference method is employed to solve these transformed problems.
It should be pointed out that, to the best of our knowledge, the existing finite difference methods based on the front-fixing approach generally impose rather restrictive conditions on the time step-size such as $\Delta t \le C(\Delta z)^2$. However, by appropriately constructing the finite difference schemes to make the coefficient matrix be $\mathrm{M}$-matrices, we only need a moderate restriction $\Delta t<{1}/{L}$ to ensure the positivity of the numerical solution and the monotonicity of the numerically spreading front.
In addition, by introducing new discrete weighted norms, the conditional stability of the proposed numerical schemes are established.
Finally, three numerical examples  are given to confirm the theoretical results. Besides, the numerical evidence for the existence of the thresholds
$\mu^*$ and $\sigma^*$ reported in Refs. \cite{du2010spreading,du2015spreading,du2011spreading,Du&Matsuzawa2015Spreading}
obtained from the numerical tests also further validates the effectiveness of the proposed numerical schemes.



\begin{thebibliography}{00}
\addcontentsline{toc}{section}{References}

\bibitem{stefan1891theorie}J. Stefan, \"Uber die Theorie der Eisbildung, insbesondere \"uber die Eisbildung im Polarmeere, Ann. Phys. 278 (1891) 269--286.

\bibitem{crank1984free}J. Crank, Free and Moving Boundary Problems, Clarendon Press, Oxford, 1984.

\bibitem{chen2015free}G.Q. Chen, H. Shahgholian, J.L. Vazquez, Free boundary problems: the forefront of current and future developments, Philos. Trans. A Math. Phys. Eng. Sci. 373 (2015) 20140285.

\bibitem{carinci2016free}G. Carinci, A. De Masi, C. Giardin{\`a}, E. Presutti, Free boundary problems in PDEs and particle systems, Springer, Berlin, 2016.

\bibitem{bansch2023interfaces}E. Bänsch, K. Deckelnick, H. Garcke, P. Pozzi, Interfaces: modeling, analysis, numerics, Springer, Berlin, 2023.

\bibitem{du2010spreading}Y. Du, Z. Lin, Spreading-vanishing dichotomy in the diffusive logistic model with a free boundary, SIAM J. Math. Anal. 42 (2010) 377--405.

\bibitem{kaneko2011free}Y. Kaneko, Y. Yamada, A free boundary problem for a reaction-diffusion equation appearing in ecology, Adv. Math. Sci. Appl. 21 (2011) 467--492.

\bibitem{du2015spreading}Y. Du, B. Lou, Spreading and vanishing in nonlinear diffusion problems with free boundaries, J. Eur. Math. Soc. 17 (2015) 2673--2724.

\bibitem{du2011spreading}Y. Du, Z. Guo, Spreading-vanishing dichotomy in a diffusive logistic model with a free boundary, II, J. Differ. Equ. 250 (2011) 4336--4366.

\bibitem{Du&Matsuzawa2015Spreading}Y. Du, H. Matsuzawa, M. Zhou, Spreading speed and profile for nonlinear Stefan problems in high space dimensions, J. Math. Pures Appl. 103 (2015) 741--787.

\bibitem{caginalp1986analysis}G. Caginalp, An analysis of a phase field model of a free boundary, Arch. Ration. Mech. Anal. 92 (1986) 205--245.

\bibitem{marshall1986front}G. Marshall, A front tracking method for one-dimensional moving boundary problems, SIAM J. Sci. Stat. Comput. 7 (1986) 252--263.

\bibitem{womble1989front}D.E. Womble, A front-tracking method for multiphase free boundary problems, SIAM J. Numer. Anal. 26 (1989) 380--396.

\bibitem{hou1995numerical}T.Y. Hou, Numerical solutions to free boundary problems, Acta Numer. 4 (1995) 335--415.

\bibitem{duffy2013finite}D.J. Duffy, Finite difference methods in financial engineering: a partial differential equation approach, John Wiley \& Sons, Chichester, 2006.

\bibitem{piqueras2017front}M.-A. Piqueras, R. Company, L. Jódar, A front-fixing numerical method for a free boundary nonlinear diffusion logistic population model, J. Comput. Appl. Math. 309 (2017) 473--481.

\bibitem{kumar2020moving}A. Kumar, A moving boundary problem with space-fractional diffusion logistic population model and density-dependent dispersal rate, Appl. Math. Model. 88 (2020) 951--965.

\bibitem{liu2020krylov}S. Liu, X. Liu, Krylov implicit integration factor method for a class of stiff reaction-diffusion systems with moving boundaries, Discrete Contin. Dyn. Syst. Ser. B 25(1) (2020) 141--159.

\bibitem{liu2020numerical}S. Liu, Y. Du, X. Liu, Numerical studies of a class of reaction-diffusion equations with Stefan conditions, Int. J. Comput. Math. 97 (2020) 959--979.

\bibitem{casaban2023qualitative}M.C. Casabán, R. Company, V.N. Egorova, L. Jódar, Qualitative numerical analysis of a free-boundary diffusive logistic model, Mathematics 11(6) (2023) 1296.

\bibitem{yang2019numerical}L. Yang, L. Bao, Numerical study of vanishing and spreading dynamics of chemotaxis systems with logistic source and a free boundary, Discrete Contin. Dyn. Syst. Ser. B 26(2) (2021) 1083--1109.

\bibitem{casaban2024random}M.-C. Casabán, R. Company, V.N. Egorova, L. Jódar, A random free-boundary diffusive logistic differential model: Numerical analysis, computing and simulation, Math. Comput. Simul. 221 (2024) 55--78.

\bibitem{landau1950heat}H.G. Landau, Heat conduction in a melting solid, Quart. Appl. Math. 8 (1950) 81--94.

\bibitem{berman1994nonnegative}A. Berman, R.J. Plemmons, Nonnegative Matrices in the Mathematical Sciences, SIAM, Philadelphia, 1994.

\bibitem{horn2012matrix}R.A. Horn, C.R. Johnson, Matrix Analysis, Cambridge Univ. Press, Cambridge, 2012.
\end{thebibliography}
\end{document}